\documentclass[12pt,leqno, draft]{amsart}
\usepackage{amsmath,color}
\usepackage{booktabs}
\usepackage{tikz} 
\usepackage[english]{babel}
\usepackage{bm}
\usepackage{diagbox} 

\usepackage[authoryear]{natbib}

\newcommand{\R}{{\mathbb R}}
\newcommand{\N}{{\mathbb N}}

\def\BbbP{{{\rm I} \!  {\rm P}}}
\def\BbbE{{{\rm I}  \! {\rm E}}}
\newcommand{\Var}{{\rm Var}}

\newcommand{\claw}{\stackrel{\mathcal{D}}{\longrightarrow}}
\newcommand{\cpr}{\stackrel{\mathcal{P}}{\longrightarrow}}

\def \P{\BbbP}
\def \E{\BbbE}

\newtheorem{theorem}{Theorem}
\newtheorem{lemma}{Lemma}[section]
\newtheorem{corollary}[lemma]{Corollary}
\theoremstyle{definition}
\newtheorem{remark}{Remark}

\begin{document}

\title[Weighted Test Statistics]
{Power of Weighted Test Statistics for Structural Change in Time Series} 
\author[H. Dehling]{Herold Dehling}
\author[K. Vuk]{Kata Vuk}
\author[M. Wendler]{Martin Wendler}
\today
\address{
Fakult\"at f\"ur Mathematik, Ruhr-Universit\"at Bo\-chum,  Universit\"atsstra\ss e 150,
44780 Bochum, Germany}
\email{herold.dehling@ruhr-uni-bochum.de}
\address{
Fakult\"at f\"ur Mathematik, Ruhr-Universit\"at Bo\-chum, Universit\"atsstra\ss e 150,
44780 Bochum, Germany}
\email{kata.vuk@ruhr-uni-bochum.de}
\address{Institut f\"ur Mathematische Stochastik, Otto-von-Guericke-Universit\"at Magdeburg, Universit\"atsplatz 2, 39106 Magdeburg, Germany}
\email{martin.wendler@ovgu.de}


\maketitle

\begin{abstract}
We investigate the power of some common change-point tests as a function of the location of the change-point. The test statistics are maxima of weighted U-statistics, with the CUSUM test and the Wilcoxon change-point test as special examples. We study the power under local alternatives, where we vary both the location of the change-point and the magnitude of the change. We quantify in which way weighted versions of the tests are more powerful when the change occurs near the beginning or the end of the time interval, while losing power against changes in the center. 
\end{abstract}

\section{Introduction} 
In this paper, we will compare the power of some standard change point tests when the weight functions vary. We will consider alternatives where a jump occurs in the center of the observation period as well as alternatives where a jump occurs very early or very late. 
According to the change-point folklore, very early and very late changes are better detected by tests with weights that increase near the boundary of the observation period.
In this paper, we aim to shed some light on this problem, both by precise mathematical results and by simulations. We will do so by considering local alternatives that express the phenomenon of change points near the border of the observation period. 
Our results indicate that optimal weights depend on the rate at which the change-point converges to the border of the observation time. 

We investigate the model of at most one change, 
assuming that the data are generated by the signal plus noise model
\begin{equation}
X_i=\mu_i+\xi_i,\; i\geq 1,
\label{eq:basic-model}
\end{equation}
where $(\mu_i)_{i\geq 1}$ is an unknown signal, and where
$(\xi_i)_{i\geq 1}$ is a mean zero i.i.d.\ process. Based on the observations $X_1,\ldots, X_n$, we want to test the hypothesis
\[
  H:\; \mu_1=\ldots=\mu_n
\]
that there is no change in the location during the observation period $\{1,\ldots,n\}$ against the alternative that there is a change at some unknown point in time $k^\ast$, i.e.
\[
  A:\; \mu_1=\ldots=\mu_{k^\ast} \neq \mu_{k^\ast+1} =\ldots=\mu_n, \mbox{ for some } k^\ast \in \{1,\ldots, n-1\}.
\]
We will specifically consider alternatives where the location $k^\ast$ as well as the height $\Delta=\mu_{k^\ast+1} - \mu_{k^\ast}$ of the change is allowed to vary with the sample size. Thus, strictly speaking, our model is a triangular array where the signal is given by $(\mu_{n,i})_{1\leq i\leq n, n\geq 1}$. We study U-process based test statistics defined as
\[
  \max_{1\leq k < n} \frac{1}{ n^{3/2} \big(  \frac{k}{n}(1-\frac{k}{n}) \big)^\gamma  } \sum_{i=1}^k \sum_{j=k+1}^n h(X_i,X_j),
\]
where $0\leq \gamma < \frac{1}{2}$ is a tuning parameter and where $h:\mathbb{R}^2 \rightarrow \mathbb{R}$ is a kernel function. We consider kernels of the type 
\[h(x,y)=g(y-x),\]
where $g$ is an odd function. That type of kernel function covers many test statistics including CUSUM and Wilcoxon. The parameter $\gamma$ defines the strength of the weight near the borders of the observation period. The greater $\gamma$ is, the higher are the weights at the border. The limit $\gamma=\frac{1}{2}$ is exceptional because in this case the test statistics asymptotic distribution is an extreme value distribution. The choice $\gamma=0$ yields the non-weighted version of the test.

The asymptotic distribution of the U-process based test statistic with kernel $h(x,y)$ has been studied by various authors, both for i.i.d.\ data and for data with serial correlations. For i.i.d.\ data, the theory is summarized in the seminal monograph by \cite{CsH:1997}, where also dependent data are treated in connection with the CUSUM test. 
For short-range dependent data and general $U$-statistic based tests, \cite{DFGW:2015} investigated the case $\gamma=0$, and \cite{DVW:2022} treated the special weighted case $\gamma=\frac{1}{2}$.  
Most research on change-point tests has been devoted to the distribution  of test statistics under the null hypothesis of no change. 
In the case of i.i.d.\ data, \cite{S:1991} studied non-weighted $U$-statistics processes for detecting a change in the distribution under contiguous alternatives, where the location of the change point is in the center of the observation period.  
Again for i.i.d.\ data, \cite{F:1994} studied the power of non-weighted tests based on anti-symmetric U-statistics under local alternatives, allowing both the location and the height of the jump to vary with the sample size, including possible jumps near the boundary of the observation period.
For long-range dependent processes, \cite{DRT:2013} and \cite{DRT:2017} studied the asymptotic distribution of the Wilcoxon test statistic under the hypothesis as well as under the alternative.  
\cite{HRZ:2021} studied the norms of weighted functional CUSUM processes and derived the asymptotic distribution under the hypothesis of no change as well as under local alternatives in the presence of a change in in the covariance. 
\cite{RGLA:2011} considered weighted versions of CUSUM tests for detecting early changes, and compared various tests via simulation studies.
\cite{HMR:2020} investigated the power of non-weighted CUSUM tests under local alternatives for short-range dependent data. Horv\'ath et al.\ also investigated the power of so-called R\'enyi change-point tests, which are CUSUM tests weighted even heavier at the end of the observation period. Related to weighted and non-weighted U-statistics based processes, \cite{RW:2020} address epidemic changes and \cite{BGH:2009} address changes in mean of the covariance structure of a linear process. \cite{G:2000b} compared the power of U-statistic based change-point tests for the online and offline scenario and described the large sample behavior of these tests under local alternatives. However, these considerations don't include the contiguous alternative with $O(n^{-1/2})$ size changes. In a Monte Carlo simulation study, \cite{XLX:2014} investigated how the change-point location influences the ability of Wilcoxon based Pettitt test.

In the present paper, we investigate changes where the change point occurs on the scale of $n^\kappa$, for different values of $\kappa\in (0,1]$. In this way, for $\kappa=1$, one obtains changes in the center of the observation period, while $0<\kappa<1$ corresponds to very early changes. 
We study both the case of fixed size jumps and jumps whose size decreases with increasing sample size.  
In addition to comparing power functions of different tests, we study the power function in relation to the envelope power. The envelope power is defined as the maximal power that can be achieved by testing the hypothesis of no change against a fixed alternative $(k,\Delta)$, where $k$ denotes the location and $\Delta$ the height of the change. 

\section{Preliminary remarks and definitions}

Before we present our main results in the next sections, we introduce some notations. We consider the signal plus noise model \eqref{eq:basic-model} and test the hypothesis \[H:\; \mu_1=\ldots=\mu_n \] against the alternative \[A:\; \mu_1=\ldots=\mu_{k^\ast} \neq \mu_{k^\ast+1} =\ldots=\mu_n, \mbox{ for some } k_n^\ast \in \{1,\ldots, n-1\}.\] The next sections attend to two different types of alternative. 

The section \textit{Small change after fix proportion of time} deals with local alternatives in which the time of change is proportional to the sample size and the jump height decreases as the sample size increases. We call this alternative $A_1$ and define more precisely
 \begin{align*}
 A_1: \; \mu_1=\ldots=\mu_{k_n^\ast} \neq \mu_{k_n^\ast+1} =\ldots=\mu_n, \text{ with } k_n^*=[\tau^*n] \text{ and } \Delta_n=\mu_{k_n^*+1} -\mu_{k_n^*}= \frac{c}{\sqrt{n}}, 
 \end{align*}
where $\tau^* \in (0,1)$ and $c$ is a constant. 

In the section \textit{Early change with fixed height}, we consider another type of alternative in which the jump height is kept constant, while the time of change moves closer to the border of the observation range. We model this alternative as follows
 \begin{align*}
 A_2: \; &\mu_1=\ldots=\mu_{k_n^\ast} \neq \mu_{k_n^\ast+1} =\ldots=\mu_n, \text{ with } k_n^* \approx cn^{\kappa}, \text{ meaning that } \frac{k_n^*}{cn^{\kappa}}\rightarrow 1,\\ &\text{ and } \Delta_n=\mu_{k_n^*+1} -\mu_{k_n^*} \equiv \Delta, 
 \end{align*}
where $c$ is a constant and where the parameter $\kappa$ is defined as 
\begin{align*}
\kappa=\frac{1-2\gamma}{2(1-\gamma)},~\gamma\in [0,\tfrac{1}{2}).
\end{align*} 
Note that by definition $\kappa \in (0,\frac{1}{2}].$ 

In short, we can write the corresponding model as 
\begin{align} \label{model}
 X_i=\left\{
    \begin{array}{ll}
      \mu+ \xi_i & \mbox{ for } i\leq k_n^\ast \\[2mm]
      \mu+\Delta_n+\xi_i & \mbox{ for } i\geq k_n^\ast+1,   
    \end{array}
  \right.
\end{align}
where $(\xi_i)_{\geq1} $ is a mean zero i.i.d.\ process and where $k_n^*$ and $\Delta_n$ are chosen as in $A_1$ or $A_2$. 
\begin{remark}
The specific choice of $\kappa$ in $A_2$ leads to a non trivial limit distribution under the alternative.
\end{remark}
In order to test $H$ vs.\@ $A_i,~i\in \{1,2\}$, we use the test statistic
\begin{align}
G^{\gamma}_n(k):= \frac{1}{n^{3/2} \big( \frac{k}{n} (1-\frac{k}{n}\big)\big)^\gamma } \sum_{i=1}^k \sum_{j=k+1}^n g(X_j-X_i), \label{wgeneral} 
\end{align}
where $\gamma \in [0,1/2)$ and where $g$ is an odd function, i.e.\ $g(-x)=-g(x).$ Note that the case $\gamma=0$ refers to the unweighted test statistic. We determine the limiting distribution of the test statistic under the the alternatives $A_1$ and $A_2$. In that proceeding, slightly different terms appear depending on whether $k\leq k_n^*$ or $k\geq k_n^*$. For the sake of simplicity we combine these terms into one function. We define $\phi_n: \{1,\ldots,n \} \rightarrow \N$ by
\[
  \phi_n(k):=\left\{
  \begin{array}{ll}
    k (n-k_n^\ast) & \mbox{ for } k\leq k_n^\ast \\[2mm]
    k_n^\ast (n-k) & \mbox{ for } k\geq k_n^\ast,
  \end{array}
  \right.
\] 
and analogously the continuous version $\phi_{\tau^\ast}:[0,1] \rightarrow \R$ by
\[
  \phi_{\tau^\ast}(\lambda)=\left\{  
  \begin{array}{ll}
   \lambda(1-\tau^\ast) & \mbox{ for } \lambda \leq \tau^\ast \\
   \tau^\ast(1-\lambda) & \mbox{ for } \lambda \geq \tau^\ast.
  \end{array}
  \right. 
\]   
For later use, we denote a Wiener process by $\{W(\lambda), 0 \leq \lambda \leq 1 \}$ and a Brownian bridge process by $\{W^{(0)}(\lambda), 0 \leq \lambda \leq 1 \}$. 

\section{Small changes after fixed proportion of a sample} \label{sec:A1}
In this section we establish the asymptotic distribution of $\max_{1\leq k < n } G_n^{\gamma}(k)$ under the alternative $A_1$, i.e.\ where $k_n^*=[\tau^*n]$ and $\Delta_n=\frac{c}{\sqrt{n}}$. For the special CUSUM and Wilcoxon kernel functions, the results are stated in Corollary \ref{cusum_thm:A1_G} and \ref{wilcoxon_thm:A1_G}.

\begin{theorem} \label{thm:A1_G}
We consider model \eqref{model} under $A_1$. Assume that $g(\xi_2-\xi_1)$ has finite second moments. Moreover, assume that $\Var(h_1(\xi_1))\rightarrow 0$ and that $c_g= \lim_{n \rightarrow \infty} \sqrt{n} u(\Delta_n)$ exists. Then, for $0\leq \gamma < \frac{1}{2}$ and as $n\rightarrow \infty$,
\begin{align*}
\max_{1\leq k < n } G_n^{\gamma} (k) \claw \sup_{0 \leq \lambda \leq 1} \frac{1}{(\lambda (1-\lambda))^{\gamma}} [\sigma W^{(0)}(\lambda)+c_g \phi_{\tau^*}(\lambda)], 
\end{align*} 
where $\sigma^2= \E (g_1^2(\xi_1))>0$ and
\begin{align*}
g_1(x) &= \E[ g(\xi-x)] - \E[g(\xi-\eta)], \\
u(\Delta_n) & =\E[g(\xi-\eta+\Delta_n)-g(\xi-\eta)], \\
 h_1(x) &= \E[ g(\xi-x+\Delta_n) -g(\xi-x) ] -u(\Delta_n),
\end{align*}
where $\xi$ and $\eta$ are independent and have the same distribution as $\xi_1$. 
\end{theorem}

\begin{remark}
(i) $h_1$ and $u$ are obtained from Hoeffding's decomposition, applied to the kernel 
$ h(x,y) =g(y-x+\Delta_n)-g(y-x).$ More details are given in the proof of Theorem \ref{thm:A1_G} in Section \ref{sec:proofs}. \\ 
(ii) For $c=0$ we obtain the limit under the null hypothesis of stationarity. In order to calculate the asymptotic critical values, we need to determine the quantiles of the distribution of
\[
  \sup_{0\leq \lambda \leq 1} \frac{1}{(\lambda(1-\lambda))^\gamma} W^{(0)}(\lambda).
\]
The $\alpha$-quantiles, $\alpha\in (0.01,0.05,0.1)$, for various choices of $\gamma$, are tabulated in Table \ref{table.quantiles}.  
\end{remark}


\begin{table}[h]
\begin{tabular}{|l|ccc|}
\hline
\diagbox{$\gamma$}{$\alpha$}
		  & 0.1    & 0.05   & 0.01     \\ 
\hline 
 0        & 1.05   & 1.20   & 1.51 \\
 0.1      & 1.24   & 1.41   & 1.72 \\
 0.2      & 1.45   & 1.63   & 2.05 \\
 0.3      & 1.75   & 1.96   & 2.40 \\
 0.4      & 2.10   & 2.31   & 2.83 \\ 
 \hline
\end{tabular}
\vspace{2mm}
\caption[Table]{$\alpha$-Quantiles of $\sup_{0\leq \lambda \leq 1} \frac{1}{(\lambda(1-\lambda))^\gamma} W^{(0)}(\lambda) $ for different parameters $\gamma$, based on 10,000 repetitions.} 
\label{table.quantiles} 
\end{table}

Theorem \ref{thm:A1_G} covers both the CUSUM and Wilcoxon test statistic. Choosing $g(x)=x$ leads to the CUSUM test statistic and satisfies the assumptions. We have $u(\Delta_n)=E[\Delta_n]=\Delta_n$ and $c_g=\lim_{n\rightarrow \infty} \sqrt{n} \Delta_n=c$, as $\Delta_n=\frac{c}{\sqrt{n}}$. Moreover, \[
h_1(x)=\E[g(\xi-x+\Delta_n)-g(\xi-x)]-u(\Delta_n) = \E[\xi-x+\Delta_n-(\xi-x)]-\Delta_n = 0.
\] Thus, we can deduce the following corollary for the weighted CUSUM test statistic.
\begin{corollary} \label{cusum_thm:A1_G}
Under the assumptions of Theorem \ref{thm:A1_G}, it holds 
\begin{align*}
\max_{1\leq k < n} \frac{1}{n^{3/2} \big( \frac{k}{n} (1-\frac{k}{n}\big)\big)^\gamma } \sum_{i=1}^k \sum_{j=k+1}^n (X_j-X_i) \claw
  \sup_{0\leq \lambda \leq 1} \frac{1}{(\lambda(1-\lambda))^\gamma} \left[ \sigma W^{(0)}(\lambda) +c\phi_{\tau^\ast}(\lambda)  \right] ,
\end{align*}
where $\sigma^2=\Var(\xi_1)< \infty$.
\end{corollary} 
 To obtain the Wilcoxon test statistic, choose $g(x)=1_{\{0 \leq x\}}- \frac{1}{2}$. Then
\begin{multline*}
u(\Delta_n) =\E [ 1_{\{ \eta - \Delta_n \leq \xi \}}- 1_{\{\eta \leq \xi\}}] =\E[1_{\{ \eta - \Delta_n < \xi \leq \eta \}}]= P(\eta - \Delta_n < \xi \leq \eta)  \\\
= \int_{\mathbb{R}} (F(y)-F(y-\Delta_n)) dF(y) \approx -\Delta_n \int_{\mathbb{R}} f^2(y)dy, 
\end{multline*}
where $F$ is the distribution function and $f$ the density function of $\xi.$ This yields $c_g= c \int_{\mathbb{R}} f^2(y)dy.$ Furthermore, 
\begin{multline*}
 |h_1(x)| =|\E [1_{\{ 0 \leq \xi-x+\Delta_n\}}-1_{\{0 \leq \xi-x \}}]-u(\Delta_n) |= |\P(x-\Delta_n< \xi \leq x)-u(\Delta_n) |\\\ 
= |F(x)-F(x-\Delta_n)-u(\Delta_n) |= \big| \Delta_n \frac{F(x)-F(x-\Delta_n)}{\Delta_n}-u(\Delta_n) \big| \approx |\Delta_n f(x)-u(\Delta_n)| \\\
= \big| \Delta_n \Big( \int_{\mathbb{R}} f^2(y)dy-f(x) \Big) \big| \leq |c \Delta_n|,
\end{multline*} where $c$ is a finite constant if the density is bounded. Thus, $\Var(h_1(\xi_1))\rightarrow 0$. As all required assumptions are satisfied, we derive the following corollary.
 \begin{corollary} \label{wilcoxon_thm:A1_G}
Assume that $\xi_1$ has bounded density. Under the assumptions of Theorem \ref{thm:A1_G} it holds 
\begin{multline*}
\max_{1\leq k < n} \frac{1}{n^{3/2} \big( \frac{k}{n} (1-\frac{k}{n}\big)\big)^\gamma } \sum_{i=1}^k \sum_{j=k+1}^n  \Big(1_{\{X_i\leq X_j\}}  -\frac{1}{2}\Big) \\\
 \claw  \sup_{0\leq \lambda \leq 1} \frac{1}{(\lambda(1-\lambda))^{\gamma}} \left[ \frac{1}{\sqrt{12}} W^{(0)}(\lambda) + c\phi_{\tau^\ast}(\lambda) \int_{\mathbb{R}} f^2(y)dy \right].
\end{multline*}
\end{corollary}

\section{Early change with fixed height}

Now we consider Alternative $A_2$, i.e.\ the situation where the jump height is kept constant while the time of change moves closer to the border of the observation range. We will show that the choice of $\gamma$ influences the scales at which change points can be detected.

First, we consider the case $\gamma=0$, i.e.\ the case where the norming sequence is $\frac{1}{n^{3/2}}$ and so does not depend on $k$. For $\gamma=0$ we get $\kappa=1/2$, which yields the alternative where the change-point occurs at time $k_n^\ast \approx c \sqrt{n}$. 

\begin{theorem}\label{thm:A2_G0}
We consider model \eqref{model} under $A_2$. Assume that $g(\xi_2-\xi_1)$ has finite second moments. Moreover, assume that $\Var(h_1(\xi_1))<\infty$. Then, for $\gamma=0$ and as $n\rightarrow \infty$,
\begin{align*}
\max_{1\leq k < n} |G^0_n(k)|
 \claw  \sup_{0\leq \lambda \leq 1} \Big| \sigma W^{(0)}(\lambda) + c(1-\lambda) u(\Delta) \Big|,
\end{align*}
where $\sigma$ and $u(\Delta_n)$ and $h_1(\xi_1)$ are defined as in Theorem \ref{thm:A1_G}, albeit with $\Delta_n \equiv \Delta.$
\end{theorem}

For the CUSUM kernel we have $u(\Delta)=\Delta$ and $\Var(h_1(\xi_1))=0$. For the Wilcoxon kernel, we get $u(\Delta)=\P(0 \leq \xi_2-\xi_1\leq \Delta)$ and $\Var(h_1(\xi_1))\rightarrow 0$. Thus, we can deduce the following corollaries.
 
\begin{corollary}\label{cor:A2_cusum0}
Under the assumptions of Theorem \ref{thm:A2_G0} it holds
\[
 \max_{1\leq k < n}  \frac{1}{n^{3/2}} \Big| \sum_{i=1}^k \sum_{j=k+1}^n (X_j-X_i) \Big| \claw
 \sup_{0\leq \lambda \leq 1} | \sigma W^{(0)}(\lambda) + c(1-\lambda)\Delta |,
\]
where $\sigma^2=\Var(\xi_1)< \infty$.
\end{corollary}

\begin{corollary}\label{cor:A2_wilc0}
Under the assumptions of Theorem \ref{thm:A2_G0} it holds
\[
\max_{1\leq k < n} \frac{1}{n^{3/2}}  \Big| \sum_{i=1}^k \sum_{j=k+1}^n  \Big(1_{\{X_i\leq X_j\}}  -\frac{1}{2}\Big) \Big|  \claw \sup_{0\leq \lambda \leq 1} \Big| \frac{1}{12} W^{(0)}(\lambda) + c(1-\lambda) \P(0 \leq \xi_2-\xi_1\leq \Delta) \Big|.
\]
\end{corollary}

\begin{remark}(i) When $c=0$, i.e.\ when $k_n^\ast/\sqrt{n}\rightarrow 0$, the distribution of the test statistic under the alternative is asymptotically the same as under the null hypothesis, and thus the test has no power to detect such alternatives.  The test has asymptotically only trivial power $\alpha$, the same as the size.
\\[1mm]
(ii) The test is consistent if and only if $\lim_{n\rightarrow \infty} \frac{k_n^\ast}{\sqrt{n}} =\infty$. 
\\[1mm]
(iii) In this sense, $k_n \approx c\sqrt{n}$ is the critical time for a change-point, when one wants to obtain a consistent test. Depending on the value $c>0$, the power might asymptotically approach any value between $\alpha$ (the size) and 1. This holds, as the distribution of $\sup|\sigma W^{(0)}(\lambda)|$ is continuous and for $c \rightarrow c'$ we have
\[ \sup_{0\leq \lambda \leq 1} \big| \sigma W^{(0)}(\lambda) + c(1-\lambda) u(\Delta) \big| \rightarrow \sup_{0\leq \lambda \leq 1} \big| \sigma W^{(0)}(\lambda) + c'(1-\lambda) u(\Delta) \big| \]
in $\mathcal{D}[0,1]$. Thus, for $c \rightarrow c',$
\[ \P_{(c)} \big(\max_{1\leq k < n} |G^0_n(k)| > q_{\alpha} \big) \rightarrow \P_{(c')} \big(\max_{1\leq k < n} |G^0_n(k)| > q_{\alpha} \big), \]
where $q_{\alpha}$ is the critical value depending on the asymptotical size $\alpha$. For $c=0$ we have $ \P_{(0)} \big(\max_{1\leq k < n} |G^0_n(k)| > q_{\alpha} \big) = \alpha$ and for $c$ large enough $\P_{(c)} \big(\max_{1\leq k < n} |G^0_n(k)| > q_{\alpha} \big) = 1 .$ As the mapping 
  \[ c \mapsto \P_{(c)} \big(\max_{1\leq k < n} |G^0_n(k)| > q_{\alpha} \big) \] is continuous, it takes any value between $\alpha$ and $1$.
\end{remark}

Now, we consider the case $\gamma\in (0,1/2)$, i.e.\ where the norming sequence depends on $k$. Under the alternative $A_2$, we determine the asymptotic distribution of the test statistic $\max_{1 \leq k \leq n} G^{\gamma}_n(k)$.

\begin{theorem}\label{thm:A2_G}
We consider model \eqref{model} under $A_2$. Assume that $g(\xi_2-\xi_1)$ has finite second moments. Then, for $0<\gamma<\frac{1}{2}$ and as $n\rightarrow \infty$,
\begin{align*}
\max_{1\leq k < n} |G^{\gamma}_n(k)|
 \claw  \max\left\{c^{1-\gamma}u(\Delta), \sup_{0\leq \lambda \leq 1}\frac{\sigma}{(\lambda(1-\lambda))^\gamma} \big| W^{(0)}(\lambda)\big|\right\},
\end{align*} 
where $\sigma$ and $u(\Delta)$ are defined as in Theorem \ref{thm:A1_G}, albeit with $\Delta_n \equiv \Delta.$ 
\end{theorem}
For the special case of the CUSUM and Wilcoxon kernel we obtain the followiong corollaries.
\begin{corollary}\label{cor:A2_weighted_cusum} Under the assumptions of Theorem \ref{thm:A2_G} we obtain 
\[
   \max_{1\leq k < n}  \frac{1}{n^{3/2} \big( \frac{k}{n} (1-\frac{k}{n}\big)\big)^\gamma }  \Big| \sum_{i=1}^k \sum_{j=k+1}^n (X_j-X_i) \Big|
\claw \max\left(c^{1-\gamma}\Delta, \sup_{0\leq \lambda \leq 1}\frac{\sigma}{(\lambda(1-\lambda))^\gamma} |W^{(0)}(\lambda)|\right),
\]
where $\sigma^2=\Var(\xi_1)<\infty$.
\end{corollary}

\begin{corollary}\label{cor:A2_weighted_wilc} Under the assumptions of Theorem \ref{thm:A2_G} we obtain 
\begin{multline*}
   \max_{1\leq k < n} \frac{1}{n^{3/2} \big( \frac{k}{n} (1-\frac{k}{n}\big)\big)^\gamma } \Big| \sum_{i=1}^k \sum_{j=k+1}^n  \Big(1_{\{X_i\leq X_j\}}  -\frac{1}{2}\Big) \\\
\claw \max\left(c^{1-\gamma}\P(0 \leq \xi_2-\xi_1\leq \Delta) , \sup_{0\leq \lambda \leq 1} \frac{1}{12}\frac{1}{(\lambda(1-\lambda))^\gamma} |W^{(0)}(\lambda)|\right).
\end{multline*}
\end{corollary}

In the next Theorem we identify conditions on the limit behavior of $k_n^*/n^{\kappa}$ that guarantee consistency of the test statistic $\max_{1 \leq k < n} |G_n^{\gamma}(k)|$. We will see that the special form of the limit distribution under the local alternative results in a peculiar behavior of the asymptotic power. 

\begin{theorem}\label{theo4} 
The change point test with test statistic 
\[
  \max_{1\leq k< n} |G_n^{\gamma}(k)|
\]
is consistent if $\liminf_{n\rightarrow \infty}k_n^\ast/n^\kappa >(q_\alpha/u(\Delta))^{1/(1-\gamma)}$, there $q_\alpha$ is the critical value depending on the asymptotical size $\alpha$. In contrast, the test has asymptotically only trivial power $\alpha$ if $\limsup_{n\rightarrow \infty}k_n^\ast/n^\kappa <(q_\alpha/u(\Delta))^{1/(1-\gamma)}$.
\label{th:consistent}
\end{theorem}
\begin{proof} First note that in order to achieve asymptotic size $\alpha$, we have to choose $q_\alpha$ such that
\begin{equation*}
\P\left(\sup_{0\leq \lambda \leq 1}\frac{1}{(\lambda(1-\lambda))^\gamma} |W^{(0)}(\lambda)|>q_\alpha\right)=\alpha.
\end{equation*}
We will show that for any subseries, where exists a subsubseries $(n_j)_{j\in\N}$, such that the probabilities for $\max_{1\leq k < n}|G^{\gamma}_{n_j}(k)|>q_\alpha$ converge to 1 respectively to $\alpha$. Because the limit is the same for any subsubseries, we will then conclude that the probabilities for $\max_{1\leq k < n}|G_{n}^{\gamma}(k)|>q_\alpha$ converge to 1 respectively to $\alpha$. If $\liminf_{n\rightarrow \infty}k_n^\ast/n^\kappa >(q_\alpha/u(\Delta))^{1/(1-\gamma)}$, we can choose the subsubseries such that $k_{n_j}^\ast\approx c\cdot {n_j}^\kappa$ with $c>(q_\alpha/u(\Delta))^{1/(1-\gamma)},~ c<\infty$, so $c^{1-\gamma}u(\Delta)>q_\alpha$. So from Theorem \ref{thm:A2_G}, we know that the limit distribution of our test statistic is given by the distribution of
\begin{equation*}
\max\left(c^{1-\gamma}u(\Delta), \sup_{0\leq \lambda \leq 1}\frac{1}{(\lambda(1-\lambda))^\gamma} |W^{(0)}(\lambda)|\right),
\end{equation*}
which exceeds the critical value $q_{\alpha}$ with probability 1.

To prove the other case, note that if $\limsup_{n\rightarrow \infty}k_n^\ast/n^\kappa <(q_\alpha/u(\Delta))^{1/(1-\gamma)}$, there exists a subsubseries with $k_{n_j}^\ast\approx c\cdot {n_j}^\kappa$ for a $c<(q_\alpha/u(\Delta))^{1/(1-\gamma)}$, so $c^{1-\gamma}u(\Delta)<q_\alpha$ and for the limit distribution, it holds that
\begin{equation*}
\P\left(\max\left(c^{1-\gamma}u(\Delta), \sup_{0\leq \lambda \leq 1}\frac{1}{(\lambda(1-\lambda))^\gamma} |W^{(0)}(\lambda)|\right)>q_\alpha\right)=\alpha.
\end{equation*}
\end{proof}

\begin{remark}
It is interesting to note that for $\gamma\in(0,1/2)$, the asymptotic power is either $\alpha$ or 1, unlike in the case $\gamma=0$, where the asymptotic power can take any value in the interval $(\alpha,1)$. 
\end{remark}

\subsection{Envelope power function} 
In this section, we calculate the envelope power function for the change point problem with normal data. We determine the test that maximizes the power in any point  $(k,\Delta)$, $1\leq k\leq n-1$, $ \Delta\in \R$ in the alternative. For simplicity, we focus on the case when $\Delta>0$, and we assume that the variance is known. By the Neyman-Pearson fundamental lemma, the most powerful level $\alpha$ test for the hypothesis of no change against the alternative of a change of size $\Delta$ at time $k$ rejects the hypothesis for large values of
\[
  T_k:= \frac{1}{\sqrt{\sigma^2\left(\frac{1}{k}+\frac{1}{n-k}\right)} }  
  \left(\frac{1}{n-k}\sum_{i=k+1}^n X_i - \frac{1}{k}\sum_{i=1}^{k} X_i \right),
\]
specifically when $T_k\geq z_{1-\alpha}$, where $z_{1-\alpha}$ is the $(1-\alpha)$ quantile of the standard normal distribution.  Under the alternative $(k,\Delta)$, the test statistic $T_k$ has a normal distribution with mean $\Delta\sqrt{\frac{k(n-k)}{n \sigma^2}}$ and variance $1$. Hence, the power is given by
\[
  \P_{(k,\Delta)}(T_k\geq z_{1-\alpha})=1-\Phi\left(z_{1-\alpha}- \Delta\sqrt{\frac{k(n-k)}{n \sigma^2}}\right),
\]
where $\Phi$ denotes the standard normal density function. 
This function defines the envelope power function, i.e.\ the maximal power that can be attained by any level $\alpha$ test for the hypothesis of stationarity.

\section{Simulation study}
In this section we compare the power of the CUSUM test statistic for different values of $\gamma$ via simulations. First, we consider the setting under the first alternative $A_1$ and generate $n=1000$ independent, standard normally distributed observations with one change-point, occurring after some fraction $\tau \in (0,1)$ of time. We consider three different jump heights, namely $\Delta=\frac{5}{\sqrt{n}},~ \frac{7}{\sqrt{n}},~\frac{9}{\sqrt{n}}.$ In Figure \ref{fig_envelope_power} the size corrected power functions, together with the appropriate envelope power function, are plotted.      
\begin{figure}[h]
\resizebox{\linewidth}{!}{\input{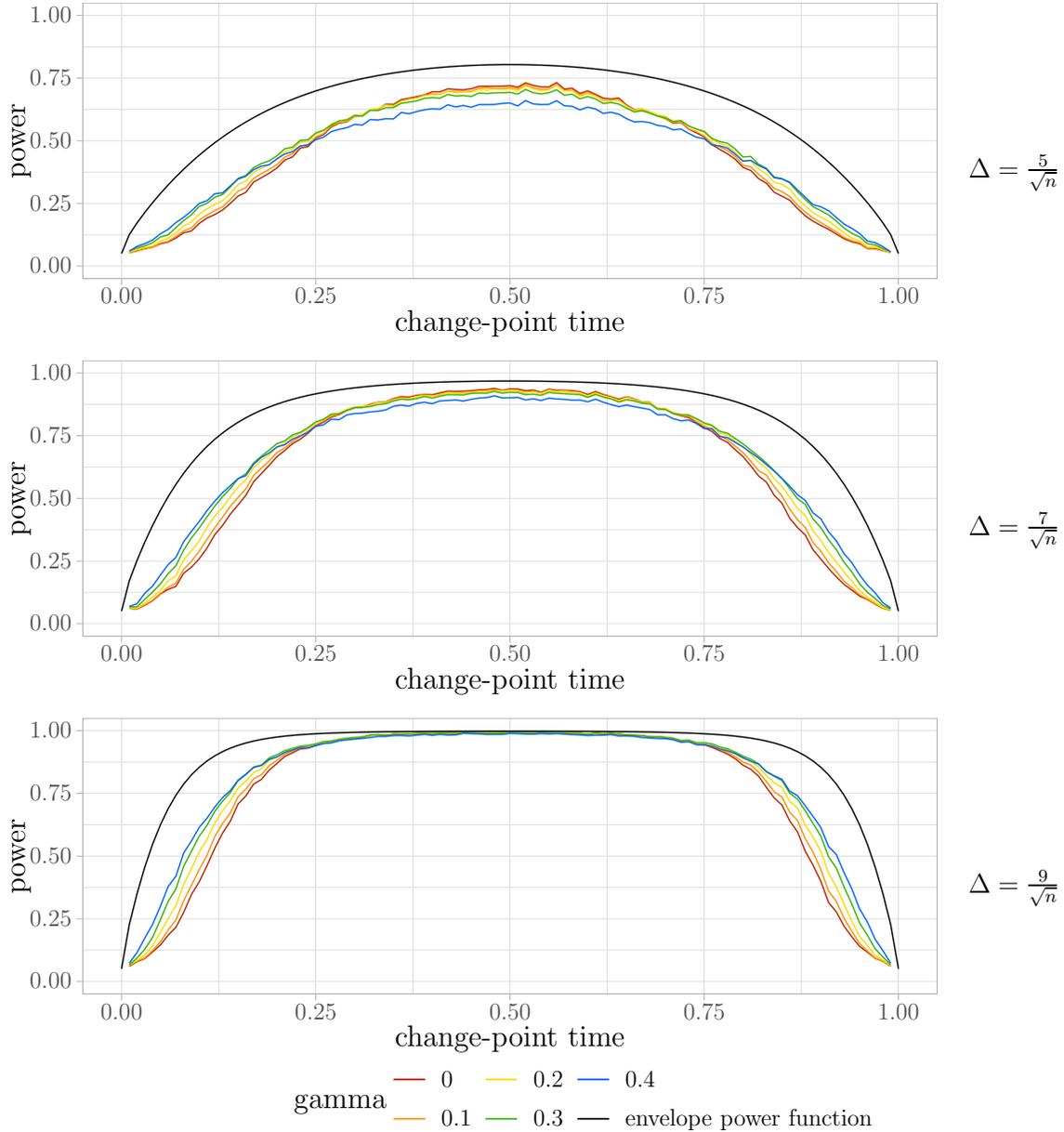}}
\caption{Size corrected power for the CUSUM test statistics and the envelope power function. The simulations are based on $n=1000$ standard normally distributed observations and 5000 runs.}
\label{fig_envelope_power}
\end{figure}
Obviously, for change-points which occur at the beginning or at the end, the power is higher the greater $\gamma$ is. If the change-point occurs around the middle of the time period, we get higher power for smaller $\gamma$. Regarding the jump heights, it is clear that the power improves for bigger jumps at each point in time. The difference in power for different $\gamma$ becomes less for higher jumps and changes that occur in the middle of the time period. For $\Delta=\frac{9}{\sqrt{n}}$ and for a change in the middle, the power is almost equal for all $\gamma$. Unlike for change-points that occur near the boundary of the time interval. In this case the difference in power for different values of $\gamma$ gets slightly greater for higher jumps.

In Figure \ref{fig_diff_power} the plots show the difference between the power of the most powerful level $\alpha=0.05$ test and the power of the weighted CUSUM tests. For higher jumps the difference is getting bigger at the boundary and smaller in the middle of the time period. For $\Delta=\frac{9}{\sqrt{n}}$ and changes in the middle, the differences in power are almost zero, i.e.\ the CUSUM test almost reaches the empirical power for all $\gamma$. For example, if we look at the middle plot in Figure \ref{fig_diff_power}, we see that for a change in the middle, we lose the most power with $\gamma=0.4$. For a change at the border, we lose the most power for $\gamma=0$.

A comparison of the overall-power is summarized in Table \ref{table.overall.power}. We have determined how much power (in $\%$) we get with the CUSUM tests, compared to the most powerful level $\alpha=0.05$ test. This basically means that we have considered the area under the curves in Figure \ref{fig_envelope_power}, where we have assumed that the area under the black curve (the envelope power function) corresponds to $100\%$ power. As before, we have considered different jump heights $\Delta=\frac{5}{\sqrt{n}},~\frac{7}{\sqrt{n}},~\frac{9}{\sqrt{n}}$. As an example, let us look at the overall-power for $\Delta=\frac{5}{\sqrt{n}}.$ The most powerful test yields $100\%$ power, and the CUSUM test with $\gamma=0.2$ yields $74.13\%$ power, which is slightly more compared to all other $\gamma$. For $\Delta=\frac{7}{\sqrt{n}}$ and $\Delta=\frac{9}{\sqrt{n}}$ the CUSUM test with $\gamma=0.4$ yields the highest overall-power.

\begin{figure}[]

\resizebox{\linewidth}{!}{\input{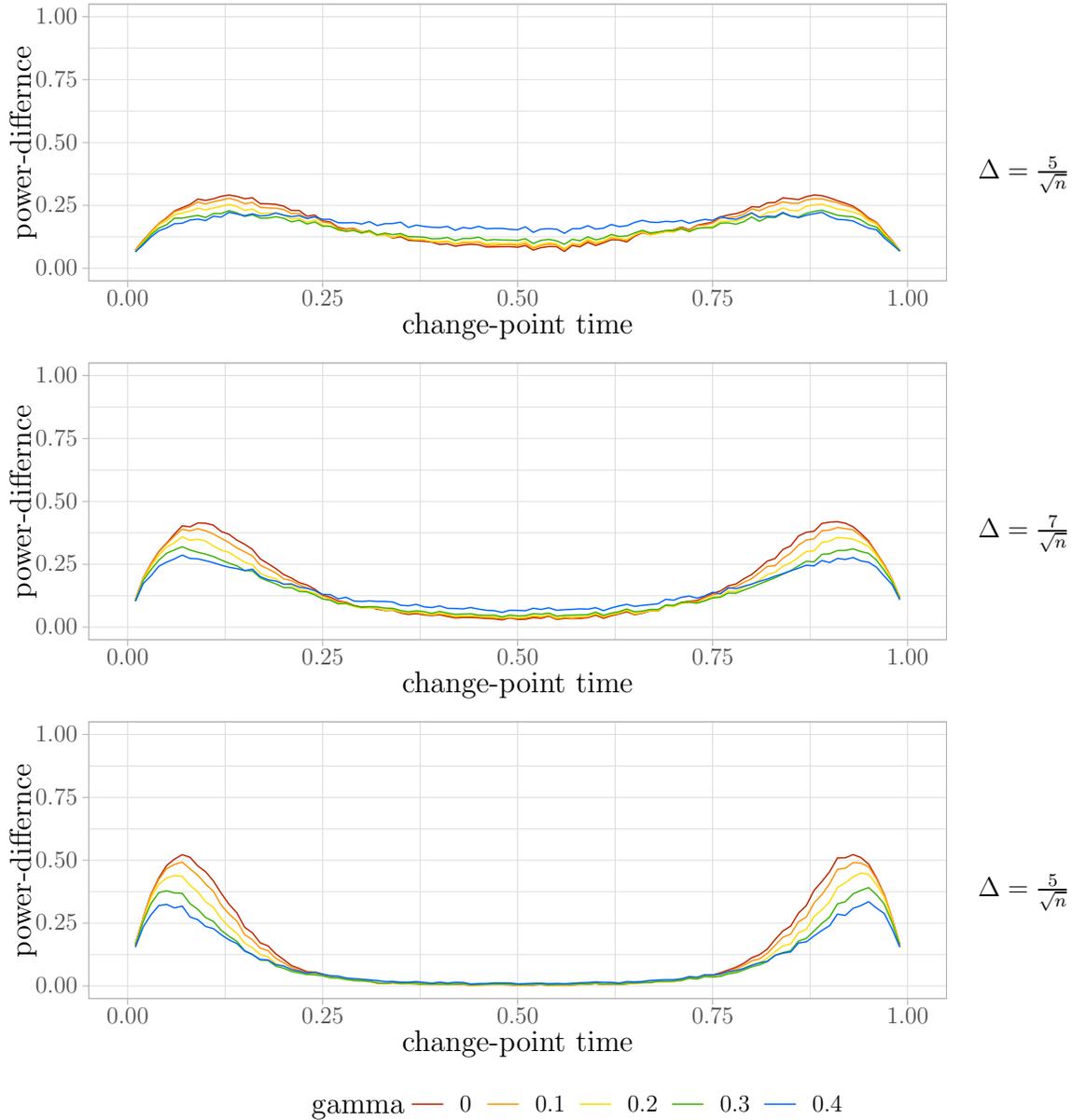}}

\caption{Power-difference: CUSUM test statistics compared to the envelope power. The simulations are based on $n=1000$ standard normally distributed observations and 5000 runs.}
\label{fig_diff_power}
\end{figure}

\vspace{0.5em}
\begin{table}[h]
\begin{tabular}{|c|ccccc|}
\hline
\diagbox{$\Delta_n$}{$\gamma$}
  						& 0          & 0.1       & 0.2 	     & 0.3	       & 0.4   \\					
 \hline \rule{0pt}{1.2\normalbaselineskip} 

 $  \frac{5}{\sqrt{n}}$    & $72.30\%$  & $72.71\%$ & $74.13\%$  & $74.75\%$   & $71.53\%$	 \\[2mm]
 $  \frac{7}{\sqrt{n}}$    & $78.97\%$  & $79.86\%$ & $81.45\%$  &	$82.65\%$  &	$81.22\%$ \\[2mm]
 $  \frac{9}{\sqrt{n}}$    & $83.98\%$  & $85.14\%$ & $86.87\%$  &	$88.52\%$  &	$89.22\%$ \\[2mm]  \hline
\end{tabular}
\vspace{2mm}
\caption[Table]{Overall-power compared to the envelope power for different values of the parameter $\gamma$ and different shift heights $\Delta_n$. The simulations are based on $n=1000$ independent, standard normally distributed observations and 5000 runs.} 
\label{table.overall.power} 
\end{table}

In Figure \ref{plot_A2} we have simulated the situation under alternative $A_2$, i.e.\ with $k_n^*\approx cn^{\kappa}$ and $\Delta_n \equiv \Delta$. The simulations are based on $n=5000$ (first plot) and $n=20000$ (second plot) standard normally distributed observations with a fixed shift height $\Delta=1$ at time $k_n^*=[ c n^{2/7} ], ~ 0<c<4.3866$. I.e.\ we consider jumps which occur very early, namely after $k_n^*=1,2,\dots, 50$ observations in the smaller sample with $n=5000$ and after $k_n^*=1,2,\dots,74$ in the larger sample $n=20000$. That's within the first $1\%$ and $0.37\%$ of the data, respectively. We compare the power functions for the weighted CUSUM test for different values of $\gamma$. In our model, the chosen $\kappa=2/7$ corresponds to $\gamma=0.3$. For smaller $\gamma$, the power converges to the level $\alpha=0.5$ and for a greater $\gamma$, the power converges to $1$.

%
%

\begin{figure}[]
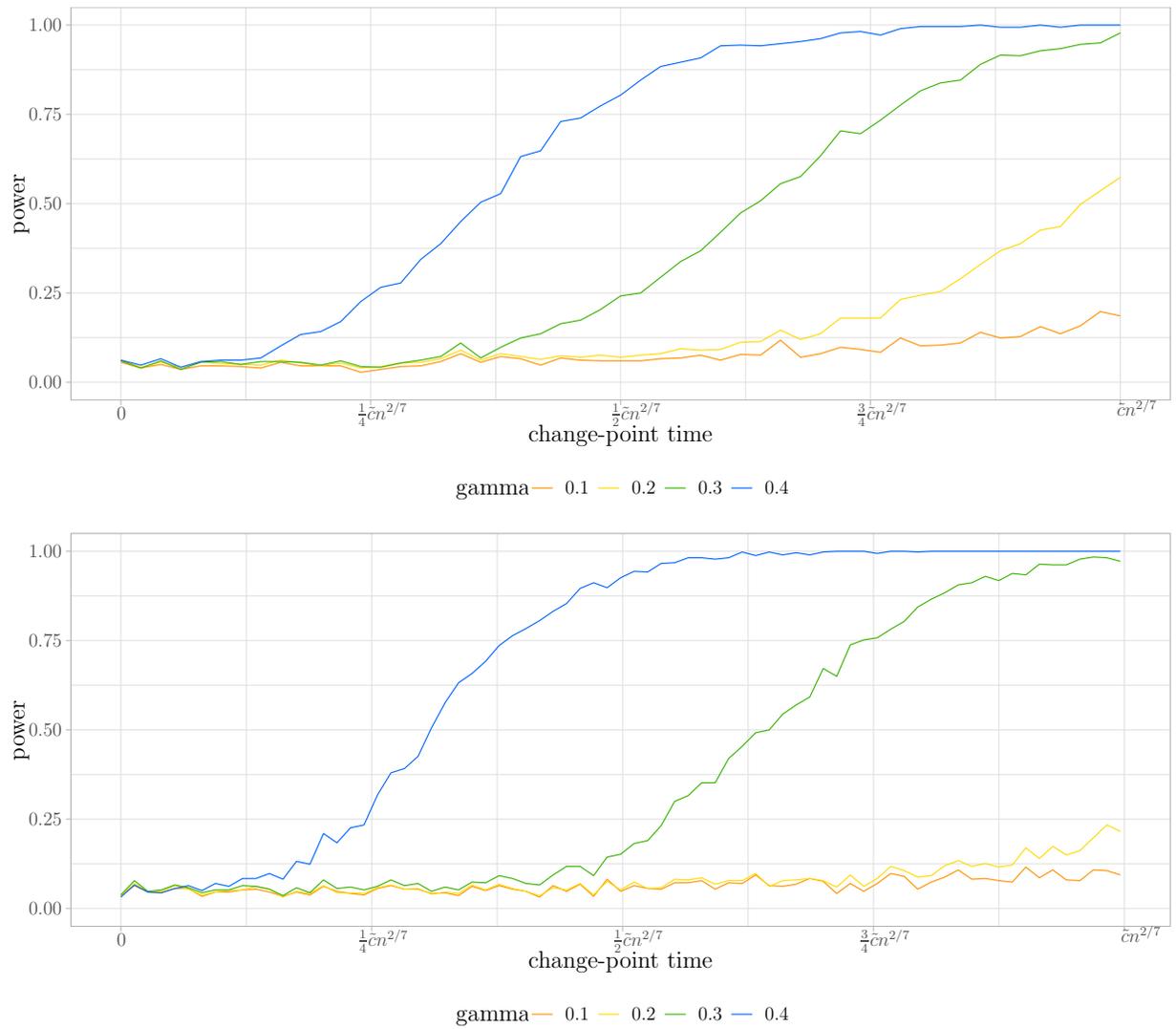

\resizebox{\linewidth}{!}{\input{Plots/plot_A2_5000.tex}}
\resizebox{\linewidth}{!}{\input{Plots/plot_A2_Kopie.tex}}
\caption{Size corrected power for the weighted CUSUM tests for $n=5000$ (top) and $n=20000$ (bottom) standard normally distributed observations with a change of size $\Delta=1$ at time $k_n^*=[c n^{2/7}]$, where $c=\tau \tilde{c}= \tau \frac{50}{5000^{2/7}},~ 0 \leq \tau \leq 1$. The simulations are based on 500 runs.  }
\label{plot_A2}
\end{figure}

%


\clearpage

\section{Proofs} \label{sec:proofs}

\subsection{Proof of Theorem \ref{thm:A1_G}}
 \begin{proof}[\unskip\nopunct]
 We recall some definitions and assumptions. We assume that $(\xi_i)_{i\geq 1}$ is an i.i.d.\ process, and that the observations are given by 
\begin{equation}
  X_i=\left\{
    \begin{array}{ll}
      \mu+ \xi_i & \mbox{ for } i\leq k_n^\ast \\[2mm]
      \mu+\Delta_n+\xi_i & \mbox{ for } i\geq k_n^\ast+1,   
    \end{array}
  \right.
 \label{eq:x-process}
\end{equation}
where $\mu$ is an unknown constant, and where $k_n^\ast=[n\tau^\ast]$, for some $\tau^\ast \in [0,1]$, and $\Delta_n=\frac{c}{\sqrt{n}}$. We consider a kernel of the type $g(y-x)$, where $g$ is an odd function, i.e.\ $g(-x)=-g(x)$. We consider the process 
\[
  G_n^{\gamma}(k)=\frac{1}{\big(\frac{k}{n} \big(1-\frac{k}{n}\big)  \big)^\gamma} \frac{1}{n^{3/2}} \sum_{i=1}^k \sum_{j=k+1}^n g(X_j-X_i). 
\]
By \eqref{eq:x-process}, we obtain the following decomposition 
\begin{align*}
  G_n^{\gamma}(k)&=\frac{1}{\big(\frac{k}{n} \big(1-\frac{k}{n}\big)  \big)^\gamma} \frac{1}{n^{3/2}} \sum_{i=1}^k \sum_{j=k+1}^n g(\xi_j-\xi_i) \\
  &\quad + \frac{1}{\big(\frac{k}{n} \big(1-\frac{k}{n}\big)  \big)^\gamma} \frac{1}{n^{3/2}} \sum_{i=1}^k \sum_{j=k+1}^n \big( g(X_j -X_i) - g(\xi_j-\xi_i) \big) \\
  &= \frac{1}{\big(\frac{k}{n} \big(1-\frac{k}{n}\big)  \big)^\gamma} \big( I_n(k) +J_n(k) \big),
\end{align*}
where the processes $I_n(k)$ and $J_n(k)$ are defined as
\begin{align*}
 I_n(k) &= \frac{1}{n^{3/2}} \sum_{i=1}^k \sum_{j=k+1}^n g(\xi_j-\xi_i) \\
 J_n(k)&= \frac{1}{n^{3/2}} \sum_{i=1}^k \sum_{j=k+1}^n \big( g(X_j-X_i)-  g(\xi_j-\xi_i) \big).
\end{align*}

We now analyze these two processes separately. Regarding $I_n(k)$, we obtain from the weighted functional central limit theorem for two-sample U-statistics that
\begin{align} \label{weighted_fclt}
\Big( \frac{1}{(\lambda (1-\lambda))^\gamma } I_n([n\lambda])  \Big)_{0\leq \lambda \leq 1}
\stackrel{\mathcal{D}}{\longrightarrow} \Big( \frac{\sigma}{( \lambda(1-\lambda))^\gamma }  W^{(0)}(\lambda)  \Big)_{0\leq \lambda \leq 1,}
\end{align}
see Theorem 2.11 in \cite{CsH:1997}. We will analyze the limit behavior of  $J_n(k)$ in two steps which we formulate as separate lemmas.

\begin{lemma} Under the conditions of Theorem \ref{thm:A1_G} 
\begin{equation}
 \max_{1\leq k< n} \frac{1}{ \big(\frac{k}{n} (1-\frac{k}{n})  \big)^\gamma } \big| J_n(k) -\E(J_n(k))\big| 
 \stackrel{P}{\longrightarrow} 0.
\label{eq:jn-conv}
\end{equation}
\label{lem:jn-conv}
\end{lemma}

\begin{lemma}
Under the conditions of Theorem \ref{thm:A1_G} 
\begin{equation}
\max_{1\leq k< n}  \frac{1}{ \big(\frac{k}{n} (1-\frac{k}{n})  \big)^\gamma } \big| \E(J_n(k)) - c_g\phi_{\tau^\ast}(\frac{k}{n}) \big| 
 \longrightarrow 0,
\end{equation}
where $c_g= \lim_{n\rightarrow \infty} \sqrt{n}\, u(\Delta_n) $.
\label{lem:ejn-conv}
\end{lemma}

\begin{proof}[Proof of Lemma~\ref{lem:jn-conv}] Observe that by definition of the process $(X_i)_{i\geq 1}$, we get
\begin{equation}
g(X_j-X_i) = \left\{ 
\begin{array}{ll}
g(\xi_j-\xi_i) & \mbox{ for } 1\leq i,~ j \leq k_n^\ast \mbox{ or } k_n^\ast+1\leq i,~j\leq n
\\[2mm]
g(\xi_j-\xi_i+\Delta_n) & \mbox{ for } 1\leq i\leq k_n^\ast, ~  k_n^\ast+1\leq j\leq n.
 \end{array}
\right.
\end{equation}
Thus we obtain
\begin{equation}
J_n(k) = \left\{ 
\begin{array}{ll}
\frac{1}{n^{3/2}} \sum_{i=1}^k \sum_{j=k_n^\ast+1}^n [g(\xi_j-\xi_i +\Delta_n ) -g(\xi_j-\xi_i)] & \mbox{ for } k\leq k_n^\ast 
\\[2mm]
\frac{1}{n^{3/2}} \sum_{i=1}^{k_n^\ast} \sum_{j=k+1}^n [g(\xi_j-\xi_i +\Delta_n ) -g(\xi_j-\xi_i)] & \mbox{ for } k\geq k_n^\ast +1. 
 \end{array}
\right.
\end{equation}
By the Hoeffding decomposition, applied to the kernel 
\[
  h(x,y) =g(y-x+\Delta_n)-g(y-x),
\]
we obtain
\begin{align*}
 u(\Delta_n) &= \E(h(\xi,\eta))=\E[g(\eta-\xi+\Delta_n)-g(\eta-\xi)  ] \\
 h_1(x) &= \E(h(x,\eta))-u(\Delta_n)= \E[ g(\eta-x+\Delta_n) -g(\eta-x) ] -u(\Delta_n) \\
 h_2(y) & = \E(h(\xi, y))-u(\Delta_n) = \E[g(y-\xi+\Delta_n)-g(y-\xi)] -u(\Delta_n) \\
 \psi(x,y)&= h(x,y)-u(\Delta_n) -h_1(x)-h_2(y),
\end{align*}
where $\xi$ and $\eta$ are two independent random variables with the same distribution as $\xi_1$. Note that by definition, 
\[
  h(\xi_i,\xi_j) =u(\Delta_n) +h_1(\xi_i) +h_2(\xi_j) +\psi(\xi_i,\xi_j)
\]
and that all the terms on the r.h.s. are mutually uncorrelated. Then we get for $k\leq k_n^\ast$ 
\begin{align}
J_n(k)-\E(J_n(k)) &= \frac{1}{n^{3/2}} \sum_{i=1}^k \sum_{j=k_n^*+1}^n  [h_1(\xi_i)+h_2(\xi_j)] +\frac{1}{n^{3/2}} 
\sum_{i=1}^k \sum_{j=k_n^*+1}^n \psi (\xi_i,\xi_j)  \label{eq:jnk-ineq} \\
&= \frac{n-k_n^\ast}{n^{3/2}} \sum_{i=1}^k h_1(\xi_i) +\frac{k}{n^{3/2}}  \sum_{i=k_n^\ast}^n h_2(\xi_i) +\frac{1}{n^{3/2}} 
\sum_{i=1}^k \sum_{j=k_n^*+1}^n \psi (\xi_i,\xi_j) \nonumber
\end{align} 
We will now analyze the three terms on the right hand side separately. Regarding the first term, using $k_n^\ast=[n\tau^\ast]$,  we obtain
\begin{align*}
 \max_{1\leq k\leq k_n^\ast} \frac{1}{\big(\frac{k}{n} (1-\frac{k}{n})\big)^\gamma } \frac{n-k_n^\ast}{n^{3/2}} 
 \big| \sum_{i=1}^k h_1(\xi_i) \big| 
 \leq (1-\tau^\ast)^{-\gamma} n^{\gamma-\frac{1}{2}}  \max_{1\leq k\leq k_n^\ast} \frac{1}{k^\gamma} \big| \sum_{i=1}^k h_1(\xi_i) \big|. 
\end{align*}
Note that  $(k^\gamma)_{1\leq k\leq k_n^\ast}$ is an increasing sequence, so that we may apply the H\'ajek-R\'enyi inequality to obtain
\begin{align*}
&P\big( \max_{1\leq k\leq k_n^\ast} \frac{1}{\big(\frac{k}{n} (1-\frac{k}{n})\big)^\gamma } \frac{n-k_n^\ast}{n^{3/2}} 
 \big| \sum_{i=1}^k h_1(\xi_i) \big| \geq \epsilon\big)  \\
 & \leq P\big( \max_{1\leq k\leq k_n^\ast} \frac{1}{k^\gamma}
 \big| \sum_{i=1}^k h_1(\xi_i) \big| \geq \epsilon (1-\tau^\ast)^\gamma n^{1/2-\gamma} \big) \\
 & \leq \frac{1}{\epsilon^2 (1-\tau^\ast)^{2\gamma} n^{1-2\gamma}}\sum_{j=1}^{k_n^\ast} \frac{1}{j^{2\gamma}} \Var(h_1(\xi)) \\
 &\leq C\,  \Var(h_1(\xi)). 
\end{align*}
This converges to zero for $n\rightarrow \infty$, as $\Var(h_1(\xi))\rightarrow 0.$ 
\\
Regarding the second term on the right hand side of \eqref{eq:jnk-ineq}, we obtain
\begin{align*}
   \max_{1\leq k\leq k_n^\ast} \frac{1}{\big(\frac{k}{n} (1-\frac{k}{n})\big)^\gamma } \frac{k}{n^{3/2}} 
 \big| \sum_{i=k_n^*+1}^n h_1(\xi_i) \big| 
 &\leq \frac{k^{1-\gamma} n^\gamma}{(1-\tau^\ast)^\gamma n^{3/2}}  \big|  \sum_{i=k_n^*+1}^n h_1(\xi_i) \big| \\
 & \leq \frac{(\tau^\ast)^{1-\gamma} }{ (1-\tau^\ast)^\gamma} \frac{1}{ n^{1/2}} \big|  \sum_{i=k_n^*+1}^n h_1(\xi_i) \big|.
\end{align*}
Hence, using Chebychev's inequality, we obtain
\begin{align*}
   P(\max_{1\leq k\leq k_n^*} \frac{1}{\big(\frac{k}{n} (1-\frac{k}{n})\big)^\gamma } \frac{k}{n^{3/2}} 
 \big| \sum_{i=k_n^*+1}^n h_1(\xi_i) \big|  \geq \epsilon) 
&  \leq  
P(
 \big| \sum_{i=k_n^*+1}^n h_1(\xi_i) \big|  \geq C \epsilon n^{1/2} )  \\
 & \leq \frac{1}{C^2 \epsilon^2 n} \Var( \sum_{i=k_n^*+1}^n h_2(\xi_i)) \leq C\, \Var(h_2(\xi_1)), 
\end{align*} 
where $\Var(h_2(\xi_1))\rightarrow 0.$
\\
Regarding the third term on the right hand side of \eqref{eq:jnk-ineq}, the process
\[
\Big(\sum_{i=1}^k\sum_{j=k_n^*+1}^n \psi(\xi_i,\xi_j)\Big)_{1\leq k \leq k_n^\ast}
\] 
is a martingale with respect to the filtration $\mathcal{F}_k= \sigma(\xi_1, \dots, \dots \xi_k, \xi_{k_n^*+1}, \dots, \xi_n)$. Clearly,  $\sum_{i=1}^k\sum_{j=k_n^*+1}^n \psi(\xi_i,\xi_j)$ is adapted to $\mathcal{F}_k$. Moreover, for $m>k$,
\begin{align*}
\E \Big(\sum_{i=1}^m \sum_{j=k_n^*+1}^n \psi(\xi_i,\xi_j) \Big| \mathcal{F}_k \Big)  &= \E \Big(\sum_{i=1}^k \sum_{j=k_n^*+1}^n \psi(\xi_i,\xi_j) \Big| \mathcal{F}_k \Big) + \E \Big(\sum_{i=k+1}^m \sum_{j=k_n^*+1}^n \psi(\xi_i,\xi_j) \Big| \mathcal{F}_k \Big) \\
&= \sum_{i=1}^k \sum_{j=k_n^*+1}^n \psi(\xi_i,\xi_j),
\end{align*}
as $\sum_{i=1}^k \sum_{j=k_n^*+1}^n \psi(\xi_i,\xi_j)$ is $\mathcal{F}_k$-measurable and 
\begin{align*}
\E \Big(\sum_{i=k+1}^m \sum_{j=k_n^*+1}^n \psi(\xi_i,\xi_j) \Big| \mathcal{F}_k \Big) = \sum_{i=k+1}^m \sum_{j=k_n^*+1}^n \E ( \psi(\xi_i,\xi_j) | \mathcal{F}_k) =  \sum_{i=k+1}^m \sum_{j=k_n^*+1}^n \E ( \psi(\xi_i,\xi_j) | \xi_j) =0,
\end{align*} 
where the last equality holds, as $\psi$ is degenerate, i.e.\ $\E(\psi(\xi_1,x))=0$.
Furthermore, we get 
$
 (\frac{k}{n} (1-\frac{k}{n}))^\gamma\geq n^{-\gamma}(1-\frac{k_n^\ast}{n}) \geq (1-\tau^\ast) n^{-\gamma},$
and hence 
\[
 \frac{1}{\big(\frac{k}{n}(1-\frac{k}{n} ) \big)^\gamma n^{3/2}} \leq C n^{\gamma -3/2}.
\]
Thus, we finally obtain from Doob's maximal inequality 
\begin{align*}
& P\Big(\max_{1\leq k\leq k_n^\ast} \frac{1}{\big(\frac{k}{n}(1-\frac{k}{n} ) \big)^\gamma n^{3/2}}  \big|\sum_{i=1}^k \sum_{j=k_n^*+1}^n \psi(\xi_i,\xi_j)\big| \geq \epsilon  \Big)   \\
&\qquad\leq P\Big(\max_{1\leq k\leq k_n^\ast}  C n^{\gamma-3/2} \big| \sum_{i=1}^k \sum_{j=k_n^*+1}^n \psi(\xi_i,\xi_j)\big| \geq \epsilon  \Big)  \\
&\qquad \leq P\Big(\max_{1\leq k\leq k_n^\ast}  \big| \sum_{i=1}^k \sum_{j=k_n^*+1}^n \psi(\xi_i,\xi_j)\big| \geq C \epsilon n^{3/2-\gamma}  \Big)  \\
&\qquad \leq \frac{1}{C\epsilon^2 n^{3-2\gamma} } \Var\Big(  \sum_{i=1}^k \sum_{j=k_n^*+1}^n \psi(\xi_i,\xi_j)\Big)\\
&\qquad \leq C n^{2\gamma -3} k_n^\ast (n-k_n^\ast) \Var(\psi(\xi_1,\xi_2)) \leq Cn^{2\gamma -1} \Var(\psi(\xi_1,\xi_2))
\end{align*}
Since $\gamma<1/2$, the right hand side converges to zero as $n\rightarrow \infty$. Hence we have shown that $\max_{1\leq k\leq k_n^\ast} \frac{1}{(\frac{k}{n}(1-\frac{k}{n}))^\gamma} |J_n(k)-\E J_n(k)| \rightarrow 0$.  In an analogous way, we can establish that $\max_{ k_n^\ast \leq k< n} \frac{1}{(\frac{k}{n}(1-\frac{k}{n}))^\gamma} |J_n(k)-\E J_n(k)| \rightarrow 0$. 
\end{proof}

\begin{proof}[Proof of Lemma~\ref{lem:ejn-conv}] Observe that 
\begin{equation}
\E(J_n(k)) = \left\{ 
\begin{array}{ll}
k(n-k_n^\ast)\frac{1}{n^{3/2}} u(\Delta_n) & \mbox{ for } k\leq k_n^\ast 
\\[2mm]
k_n^\ast (n-k)\frac{1}{n^{3/2}} u(\Delta_n) & \mbox{ for } k\geq k_n^\ast . 
 \end{array}
\right.
\end{equation}
By the definition of $\phi_n(k)$ and $\phi_{\tau^*}(\lambda)$, we obtain
\[
  \E(J_n(k)) =  \phi_n(k) \frac{1}{n^{3/2}} u(\Delta_n) =\phi_{\tau^\ast}(\frac{k}{n}) \,\sqrt{n}\,  u(\Delta_n) .
\]
Then we have 
\begin{align*}
& \max_{1\leq k< n}  \frac{1}{ \big(\frac{k}{n} (1-\frac{k}{n})  \big)^\gamma } \big| \E(J_n(k)) - c_g \, \phi_{\tau^\ast}(\frac{k}{n}) \big|  \\
= & \max_{1\leq k< n}  \frac{1}{ \big(\frac{k}{n} (1-\frac{k}{n})  \big)^\gamma } \big| \sqrt{n}\  u(\Delta_n) \, \phi_{\tau^\ast}(\frac{k}{n}) - c_g \, \phi_{\tau^\ast}(\frac{k}{n}) \big| \\
 \leq & \max_{1\leq k< n}  \frac{1}{ \big(\phi_{\tau^\ast}(\frac{k}{n})   \big)^\gamma } \big| \sqrt{n}\  u(\Delta_n) \, \phi_{\tau^\ast}(\frac{k}{n}) - c_g \, \phi_{\tau^\ast}(\frac{k}{n}) \big| \\
 =&\max_{1\leq k< n}  \Big(\phi_{\tau^\ast}(\frac{k}{n})   \Big)^{1-\gamma} \, \big| \sqrt{n}\  u(\Delta_n)  - c_g \big| \\
 = & \big(\phi_{\tau^\ast}(\tau^*)   \big)^{1-\gamma} \, \big| \sqrt{n}\  u(\Delta_n)  - c_g \big|.
\end{align*}
This converges to zero for $n \rightarrow \infty$, as $c_g=\lim_{n \rightarrow \infty} \sqrt{n} u(\Delta_n).$
\end{proof}
Now those two lemma, we can deduce the limit behavior of $J_n(k)$. Together with \ref{weighted_fclt}, we can conclude the statement of the theorem.   
\end{proof}

\subsection{Proof of Theorem \ref{thm:A2_G0}}
\begin{proof}[\unskip\nopunct]
Let $I_n(k)$, $J_n(k)$ and $u(\Delta)$ be defined as in the proof of Theorem \ref{thm:A1_G}. Set
\begin{align*}
Z_n &:=  \max_{1\leq k < n} | I_n(k) + J_n(k)| , \\
Z_{n,m} &:= \max \left\{ \max_{k \leq n/m} |J_n(k)|, \max_{k > n/m} |I_n(k)+J_n(k)| \right\} ,\\
Z_{(m)} &:= \max \left\{ |c u(\Delta)|, \sup_{\lambda > 1/m}  | \sigma W^{(0)}(\lambda) + c(1-\lambda) u(\Delta) | \right\}, \\
Z &:=  \sup_{0\leq \lambda \leq 1} | \sigma W^{(0)}(\lambda) + c(1-\lambda) u(\Delta) | .
\end{align*}
The idea is to show 
\begin{align}
Z_{n,m} \claw Z_{(m)}, ~ \text{as} ~ n\rightarrow \infty , \label{tilde:Znm-Zm}\\
Z_{(m)} \claw Z, ~ \text{as} ~ m\rightarrow \infty , \label{tilde:Zm-Z} \\
\lim_{m \rightarrow \infty } \limsup_{n \rightarrow \infty } P(|Z_{n,m}-Z_n| \geq \varepsilon)=0, \label{tilde:Znm-Zn}
\end{align}
and to deduce the convergence $Z_n \claw Z$, for $n\rightarrow \infty$, from Billingsley's triangle theorem (Theorem 3.2 in \cite{B:1999}).

First, we show \eqref{tilde:Znm-Zm}. From Lemma \ref{lem:weighted_jn-conv_A2}, we know that $\max_{1\leq k < n} |J_n(k)-\E(J_n(k))| \cpr 0$. Thus, in order to show ${\max_{k \leq n/m} |J_n(k)| \rightarrow |cu(\Delta)|}$, it suffices to show that $\max_{k \leq n/m} |\E (J_n(k))| \rightarrow |cu(\Delta)|$. As indicated in the proof of Lemma \ref{lem:ejn-conv}, we have
\begin{align*}
 \psi_n(k):= \E(J_n(k)) =   \frac{\phi_n(k)}{n^{3/2}} u(\Delta)= \left\{ 
\begin{array}{ll}
\frac{k(n-c\sqrt{n})}{n^{3/2}} u(\Delta) & \mbox{ for } k\leq k_n^\ast =c\sqrt{n}
\\[2mm]
\frac{c\sqrt{n} (n-k)}{n^{3/2}} u(\Delta) & \mbox{ for } k\geq k_n^\ast =c\sqrt{n} . 
 \end{array}
\right.
\end{align*}
As $\psi_n(k)$ is monotonically increasing for $k\leq k_n^*$ and monotonically decreasing for $k\geq k_n^*$, it takes its maximum value at $k=k_n^*=c\sqrt{n}$. We obtain
\begin{align*}
\psi_n(k_n^*)=\frac{c\sqrt{n}(n-c\sqrt{n})}{n^{3/2}} u(\Delta)=c\big(1-\frac{1}{\sqrt{n}} \big) u(\Delta) \longrightarrow cu(\Delta), \text{ as } n\rightarrow \infty.
\end{align*}
Thus, as $\frac{n}{m}> c \sqrt{n}$ for $n$ large enough, we obtain
\begin{align*}
\max_{k\leq n/m} |\E(J_n(k))| = |\psi_n(k_n^*)| \longrightarrow cu(\Delta), \text{ as } n\rightarrow \infty.
\end{align*}
Moreover, we have
\begin{align*}
\sup_{\lambda >1/m} \big| \E (J_n(\lambda n))-c (1-\lambda )u(\Delta) \big|
= \sup_{\lambda >1/m} \Big| \frac{c\sqrt{n} (n-\lambda n)}{n^{3/2}} u(\Delta) -c (1-\lambda) u(\Delta) \Big| = 0.
\end{align*}
Together with the functional central limit theorem for two-sample U-statistics, i.e.
\begin{align*}
\Big( I_n([\lambda n]) \Big)_{0 \leq \lambda \leq 1} \claw \Big(\sigma W^{(0)}([\lambda n]) \Big)_{0 \leq \lambda \leq 1},
\end{align*}
we can deduce weak convergence of the process $(I_n([\lambda n])+J_n([\lambda n]))_{(\lambda \in [1/m,1])}$ to $(\sigma W^{(0)}(\lambda)+c(1-\lambda)u(\Delta))_{\lambda \in [1/m,1]}.$ Thus we can conclude \eqref{tilde:Znm-Zm}. From the continuity of $(\sigma W^{(0)}(\lambda)+c(1-\lambda)u(\Delta))_{0\leq\lambda \leq 1}$ and as 
\begin{align*}
\max\{|cu(\Delta)|, \sup_{0\leq \lambda \leq 1 } | \sigma W^{(0)}(\lambda)+c(1-\lambda)u(\Delta)| \} = \sup_{0\leq \lambda \leq 1 } | \sigma W^{(0)}(\lambda)+c(1-\lambda)u(\Delta)|,
\end{align*}
we can deduce \eqref{tilde:Zm-Z}. For \eqref{tilde:Znm-Zn}, note that $ |Z_{n,m}-Z_n| \leq \max_{k \leq n/m} |I_n(k)|$ and
\begin{align*}
\max_{k \leq n/m} |I_n(k)| \claw \sup_{\lambda <1/m} |\sigma W^{0}(\lambda)|.
\end{align*}
Thus, we obtain
\begin{align*}
 \lim_{m \rightarrow \infty}\limsup_{n\rightarrow \infty} P(|Z_{n,m}-Z_n|\geq \epsilon) 
 \leq &  \lim_{m \rightarrow \infty}\limsup_{n\rightarrow \infty} P\Big(\max_{k \leq n/m} |I_n(k)| \geq \epsilon\Big) \\
 \leq & \lim_{m \rightarrow \infty} P\Big(\sup_{\lambda < 1/m} |\sigma W^{(0)}(\lambda)| \geq \epsilon \Big) = 0
\end{align*}
in the final step.
\end{proof}

\subsection{Proof of Theorem \ref{thm:A2_G}}
\begin{proof}[\unskip\nopunct]
Define
\begin{align*}
I_n^{\gamma}(k) &:= \frac{1}{\big(\frac{k}{n}(1-\frac{k}{n})\big)^{\gamma}} I_n(k), \\
J_n^{\gamma}(k) &:= \frac{1}{\big(\frac{k}{n}(1-\frac{k}{n})\big)^{\gamma}} J_n(k),
\end{align*}
where $I_n(k)$ and $J_n(k)$ are defined as in the proof of Theorem \ref{thm:A1_G}. Then $G_n^{\gamma}(k) = I_n^{\gamma}(k)+J_n^{\gamma}(k).$ We proceed analogously to the proof of Theorem \ref{thm:A2_G0} and define 
\begin{align*}
Z^{\gamma}_n &:=  \max_{1\leq k < n} | I_n^{\gamma}(k) + J_n^{\gamma}(k)| , \\
Z^{\gamma}_{n,m} &:= \max \left\{ \max_{k \leq n/m} |J_n^{\gamma}(k)|, \max_{k > n/m} |I_n^{\gamma}(k)+J_n^{\gamma}(k)| \right\} ,\\
Z^{\gamma}_{(m)} &:= \max \left\{c^{1-\gamma}u(\Delta), \sup_{1/m\leq \lambda \leq 1}\frac{\sigma}{(\lambda(1-\lambda))^\gamma} \big| W^{(0)}(\lambda) \big|\right\}\\
Z^{\gamma} &:= \max\left\{c^{1-\gamma}u(\Delta), \sup_{0\leq \lambda \leq 1}\frac{\sigma}{(\lambda(1-\lambda))^\gamma} \big| W^{(0)}(\lambda) \big|\right\}. 
\end{align*}
To prove that $Z^{\gamma}_n \claw Z^{\gamma}$, for $n\rightarrow \infty$, we show 
\begin{align}
Z^{\gamma}_{n,m} \claw Z^{\gamma}_{(m)}, ~ \text{as} ~ n\rightarrow \infty , \label{Znm-Zm}\\
Z^{\gamma}_{(m)} \claw {Z^{\gamma}}, ~ \text{as} ~ m\rightarrow \infty , \label{Zm-Z} \\
\lim_{m \rightarrow \infty } \limsup_{n \rightarrow \infty } P(|Z^{\gamma}_{n,m}-Z^{\gamma}_n| \geq \varepsilon)=0. \label{Znm-Zn}
\end{align}

First, we show \eqref{Znm-Zm}. From Lemma \ref{lem:weighted_jn-conv_A2} we know that for $n\rightarrow \infty$, $\max_{1\leq k < n } |\E (J_n^{\gamma}(k))-J_n^{\gamma}(k) | \rightarrow 0.$ Thus, in order to show $\max_{k\leq n/m} |J_n^{\gamma}(k)| \rightarrow c^{1-\gamma}u(\Delta)$, it suffices to show that $\max_{k\leq n/m} |\E (J_n^{\gamma}(k))| \rightarrow c^{1-\gamma}u(\Delta)$. With $k_n^*=cn^{\kappa}$ we obtain
\begin{align*}
 \E (J_n(k)) =   \frac{\phi_n(k)}{n^{3/2}} u(\Delta)= \left\{ 
\begin{array}{ll}
\frac{k(n-cn^{\kappa})}{n^{3/2}} u(\Delta) & \mbox{ for } k\leq k_n^\ast =cn^{\kappa}
\\[2mm]
\frac{cn^{\kappa} (n-k)}{n^{3/2}} u(\Delta) & \mbox{ for } k\geq k_n^\ast =cn^{\kappa}.
 \end{array}
\right.
\end{align*}
Define ${\psi}^{\gamma}_n(k):= \E (J_n^{\gamma} (k)) = \frac{1}{(\frac{k}{n} (1-\frac{k}{n}))^{\gamma}} \E (J_n(k)) =\frac{n^{2\gamma}}{k^{\gamma}(n-k)^{\gamma}} \E (J_n(k))$. Then we have 
\begin{align} \label{psigamma}
{\psi}^{\gamma}_n(k) = \left\{ 
\begin{array}{ll}
\frac{k^{1-\gamma}}{(n-k)^{\gamma}} \, n^{2\gamma-1/2}(1-cn^{\kappa-1})u(\Delta) & \mbox{ for } k\leq k_n^\ast =cn^{\kappa}
\\[2mm]
\frac{(n-k)^{1-\gamma}}{k^{\gamma}}\, cn^{2\gamma+\kappa-3/2}u(\Delta) & \mbox{ for } k\geq k_n^\ast =cn^{\kappa}.
 \end{array}
\right.
\end{align}
 ${\psi}^{\gamma}_n(k)$ is monotonically increasing for $k\leq k_n^*$ and monotonically decreasing for $k\geq k_n^*$, i.e.\ it takes its maximum at $k=k_n^* \approx cn^{\kappa}$ and  
 \begin{align*}
 {\psi}^{\gamma}_n(cn^{\kappa})= (1-cn^{\kappa-1})^{1-\gamma} \, c^{1-\gamma} \, n^{\gamma+\kappa-\kappa \gamma-1/2} u(\Delta) = (1-cn^{\kappa-1})^{1-\gamma} \, c^{1-\gamma} u(\Delta) \longrightarrow c^{1-\gamma} u(\Delta),
 \end{align*}
 as $\gamma+\kappa-\kappa \gamma-1/2=0$ and $\kappa-1<0$. For $n$ so large that $ \frac{n}{m} > cn^\kappa=k_n^*$, we obtain with the definition of $\psi^{\gamma}_n(k) = \E (J_n^{\gamma}(k))$ in \eqref{psigamma}, for $n\rightarrow \infty,$
\begin{align}
\max_{k< n/m} \big| \psi^{\gamma}_n(k)\big| &= \big| \psi^{\gamma}_n(k_n^*) \big|   \longrightarrow c^{1-\gamma} u(\Delta) , \label{psi2} \\
 \max_{k \geq n/m} \big| \psi^{\gamma}_n(k)\big| =\psi_n\Big(\frac{n}{m}\Big) &= \frac{(n-\frac{n}{m})^{1-\gamma}}{(\frac{n}{m})^{\gamma}} cn^{2\gamma+\kappa-3/2} u(\Delta)
   =c\big(1-\frac{1}{m} \big)^{1-\gamma}m^{\gamma} n^{\kappa -\frac{1}{2}} u(\Delta) \longrightarrow 0. \label{psi1}    
\end{align}
From Theorem 3 in \cite{CsSW:2008} we can deduce
\begin{equation*}
\max_{k> n/m}\big|I^{\gamma}_n(k)\big| \claw \sup_{1/m\leq \lambda \leq 1}\frac{1}{(\lambda(1-\lambda))^\gamma}| W^{(0)}(\lambda)|, \text{ as } n \rightarrow \infty.
\end{equation*}
Together with (\ref{psi2}), \eqref{psi1}, and Lemma \ref{lem:weighted_jn-conv_A2}, this implies \eqref{Znm-Zm}. Additionally, \eqref{Zm-Z} follows from the continuity of the process $ (W^{(0)}(\lambda)/(\lambda(1-\lambda))^\gamma)_{0\leq \lambda \leq 1}.$ It remains to show \eqref{Znm-Zn}. For this, note that
\begin{equation*}
\left| Z^{\gamma}_{n,m}-Z^{\gamma}_n \right| \leq \max_{k\geq n/m} \big| I_n^{\gamma}(k)\big|+ \max_{k\leq n/m}\big|J_n^{\gamma}(k) \big|.
\end{equation*}
The convergene to zero of the first summand is guaranteed by (\ref{psi1}). Using Theorem 3 in \cite{CsSW:2008}, there is a sequence of Brownian bridges $W^{(n)}$, such that
\begin{align*}
& P\Big(\Big|\max_{l\leq n/m}\big|I_n^{\gamma}(l)\big|-\sup_{\lambda \leq 1/m}\frac{1}{(\lambda(1-\lambda))^\gamma}| W^{(n)}(\lambda)|\Big|>\frac{\varepsilon}{2}\Big)\\
\leq  & P\Big(\sup_{\lambda \leq 1/m}\Big|I_n^{\gamma}([n\lambda])-\frac{1}{(\lambda(1-\lambda))^\gamma}W^{(n)}(\lambda)\Big|>\frac{\varepsilon}{2}\Big)\longrightarrow 0, \text{ as } n \rightarrow \infty.
\end{align*} 
So we can conclude that 
\begin{equation*}
\limsup_{n\rightarrow\infty}P\left(|Z^{\gamma}_{n,m}-Z^{\gamma}_n|>\varepsilon\right)\leq \limsup_{n\rightarrow\infty}P\left(\sup_{\lambda \leq 1/m}\frac{1}{(\lambda(1-\lambda))^\gamma}|W^{(n)}(\lambda)|>\frac{\varepsilon}{2}\right)\xrightarrow{m\rightarrow\infty}0,
\end{equation*}
because $W^{(n)}$ has the same distribution as $W^{(1)}$ and
\begin{equation*}
\sup_{\lambda \leq 1/m}\frac{1}{(\lambda(1-\lambda))^\gamma}|W^{(1)}(\lambda)|)\xrightarrow{m\rightarrow\infty}0
\end{equation*}
almost surely.
\end{proof}

\begin{lemma}\label{lem:weighted_jn-conv_A2}
Under the conditions of Theorem \ref{thm:A2_G0} and \ref{thm:A2_G} it holds for $0 \leq \gamma <1/2$
\begin{align*}
\max_{1 \leq k < n} \frac{1}{\big(\frac{k}{n}(1-\frac{k}{n}) \big)^{\gamma}}|J_n(k)-\E(J_n(k))| \cpr  0,~\text{as } n \rightarrow \infty,                                                                                                                             
\end{align*}
where $J_n(k)$ is defined as in the proof of Lemma \ref{lem:jn-conv}. 
\end{lemma}

\begin{proof}[Proof of Lemma \ref{lem:weighted_jn-conv_A2}]
As in the proof of Lemma \ref{lem:jn-conv}, we decompose the kernel $h(x,y)=g(y-x+\Delta)-g(y-x)$ via Hoeffding's decomposition. Then we have for $k\leq k_n^*$ 
\begin{align*}
& \frac{1}{\big(\frac{k}{n}(1-\frac{k}{n}) \big)^{\gamma}} \big( J_n(k)- \E(J_n(k)) \big) \\
= & \frac{1}{\big(\frac{k}{n}(1-\frac{k}{n}) \big)^{\gamma}} \Big( \frac{n-k_n^*}{n^{3/2}} \sum_{i=1}^k h_1(\xi_i)+\frac{k}{n^{3/2}} \sum_{i=k_n^*+1}^n h_2(\xi_i) +  \frac{1}{n^{3/2}}  \sum_{ i=1}^ {k} \sum_{j=k_n^*+1}^n \Psi(\xi_i,\xi_j) \Big).
\end{align*}
We show that the maximum of each term on the right hand side converges in probability to zero, as n goes to infinity. Recall that $k_n^*\approx cn^{\kappa},~ \kappa=\frac{1-2\gamma}{2(1-\gamma)}$ and $0<\gamma<\frac{1}{2},$ i.e.\ $0<\kappa<\frac{1}{2}.$

Regarding the first term, we get 

\begin{align*}
& \max_{1\leq k \leq k_n^* } \frac{1}{\big(\frac{k}{n}(1-\frac{k}{n}) \big)^{\gamma}} \frac{n-k_n^*}{n^{3/2}}  \big|  \sum_{i=1}^k h_1(\xi_i) \big| \\
=& \max_{1\leq k \leq k_n^* } \frac{1}{\big(1-\frac{k}{n} \big)^{\gamma}} \frac{n^{\gamma}}{k^{\gamma}} \frac{(n-cn^{\kappa})}{n^{3/2}} \big|  \sum_{i=1}^k h_1(\xi_i) \big| = \max_{1\leq k \leq k_n^* } \frac{n^{\gamma}}{\big(1-\frac{k}{n} \big)^{\gamma}}  \frac{(1-cn^{\kappa-1})}{\sqrt{n}} \frac{1}{k^{\gamma}} \big|  \sum_{i=1}^k h_1(\xi_i) \big| \\
\leq & \frac{n^{\gamma}}{\big(1-\frac{k_n^*}{n} \big)^{\gamma}}  \frac{(1-cn^{\kappa-1})}{\sqrt{n}} \max_{1\leq k \leq k_n^* }  \frac{1}{k^{\gamma}} \big|  \sum_{i=1}^k h_1(\xi_i) \big| = (1-cn^{\kappa-1})^{1-\gamma} \, n^{\gamma-1/2} \max_{1\leq k \leq k_n^* }  \frac{1}{k^{\gamma}} \big|  \sum_{i=1}^k h_1(\xi_i) \big|.
\end{align*}
As $((1/k)^{\gamma})_{1\leq k \leq k_n^*}$ is decreasing, we may apply the H\'{a}jek-R\'{e}nyi Inequlity and obtain 
\begin{align*}
& P\Big( \max_{1\leq k \leq k_n^* } \frac{1}{\big(\frac{k}{n}(1-\frac{k}{n}) \big)^{\gamma}} \frac{n-k_n^*}{n^{3/2}}  \big|  \sum_{i=1}^k h_1(\xi_i) \big| \geq \varepsilon \Big) \\
\leq & P\Big( \max_{1\leq k \leq k_n^* } \frac{1}{k^{\gamma}} \big|  \sum_{i=1}^k h_1(\xi_i) \big| \geq \varepsilon (1-cn^{\kappa-1})^{\gamma-1} \, n^{1/2-\gamma} \Big) \\
\leq & \frac{1}{\varepsilon^2 (1-cn^{\kappa-1})^{2\gamma-2} \, n^{1-2\gamma}} \sum_{i=1}^{k_n^*} \frac{1}{i^{2\gamma}} \Var(h_1(\xi_i)).
\end{align*}
As the $(\xi_i)_{i\geq 1}$ are identically distributed and as $\sum_{i=1}^{k_n^*} \frac{1}{i^{2\gamma}} \leq \int_{0}^{k_n^*} \frac{1}{x^{2\gamma}}dx $, we have 
\begin{align*}
\sum_{i=1}^{k_n^*} \frac{1}{i^{2\gamma}} \Var(h_1(\xi_i)) \leq \Var(h_1(\xi_1))\int_{0}^{k_n^*} \frac{1}{x^{2\gamma}}dx =\Var(h_1(\xi_1))(k_n^*)^{-2\gamma+1} = \Var(h_1(\xi_1)) (cn^{\kappa})^{-2\gamma+1},
\end{align*}
 where $\Var(h_1(\xi_1))$ is constant. Thus,
\begin{align*}
& P\Big( \max_{1\leq k \leq k_n^* } \frac{1}{\big(\frac{k}{n}(1-\frac{k}{n}) \big)^{\gamma}} \frac{n-k_n^*}{n^{3/2}}  \big|  \sum_{i=1}^k h_1(\xi_i) \big| \geq \varepsilon \Big)  \\
\leq  & \frac{(cn^{\kappa})^{-2\gamma+1}}{\varepsilon^2 (1-cn^{\kappa-1})^{2\gamma-2} \, n^{1-2\gamma}}  \Var(h_1(\xi_1)) \\
= & \frac{c^{-2 \gamma+1}}{\varepsilon^2} \frac{1}{(1-cn^{\kappa-1})^{2\gamma-2} \, n^{2\gamma\kappa-\kappa-2\gamma+1}} \Var(h_1(\xi_1)) \longrightarrow 0, \text{ as } n \rightarrow \infty,
\end{align*}
since by our choice of $\kappa$ we have $\kappa-1<0$ and $2\gamma \kappa-\kappa-2\gamma+1=\frac{1}{2(\gamma-1)}+1 >0$ for $0<\gamma <1/2$. For the second term we have 
\begin{align*}
& \max_{1\leq k \leq k_n^* } \frac{1}{\big(\frac{k}{n}(1-\frac{k}{n}) \big)^{\gamma}} \frac{k}{n^{3/2}}  \big|  \sum_{i=k_n^*+1}^n h_2(\xi_i) \big| \leq  \frac{(k_n^*)^{1-\gamma} n^{\gamma}}{\big(1-\frac{k_n^*}{n}\big)^{\gamma}n^{3/2}} \big|  \sum_{i=k_n^*+1}^n h_2(\xi_i) \big|  \\
= &  \frac{(cn^{\kappa})^{1-\gamma} n^{\gamma}}{(1-cn^{\kappa -1})^{\gamma} n^{3/2}} \big|  \sum_{i=k_n^*+1}^n h_2(\xi_i) \big| = \frac{(cn^{\kappa})^{1-\gamma} }{(1-cn^{\kappa -1})^{\gamma} n^{3/2-\gamma}} \big|  \sum_{i=k_n^*+1}^n h_2(\xi_i) \big|.
\end{align*}
With Chebychev's inequality, we obtain 
\begin{align*}
& P \Big( \max_{1\leq k \leq k_n^* } \frac{1}{\big(\frac{k}{n}(1-\frac{k}{n}) \big)^{\gamma}} \frac{k}{n^{3/2}}  \big|  \sum_{i=k_n^*+1}^n h_2(\xi_i) \big| \geq \varepsilon \Big)  \\
\leq & P\Big( \big|\sum_{i=k_n^*+1}^n h_2(\xi_i) \big| \geq \varepsilon \frac{(1-cn^{\kappa -1})^{\gamma} n^{3/2-\gamma}}{(cn^{\kappa})^{1-\gamma}} \Big) \leq  \frac{1}{\varepsilon^2} \frac{(cn^{\kappa})^{2-2\gamma} }{(1-cn^{\kappa -1})^{2\gamma} n^{3-2\gamma}} \Var \Big( \sum_{i=k_n^*+1}^n h_2(\xi_i) \Big) \\
 = & \frac{1}{\varepsilon^2} \frac{(cn^{\kappa})^{2-2\gamma} (n-k_n^*) }{(1-cn^{\kappa -1})^{2\gamma} n^{3-2\gamma}} \Var ( h_2(\xi_1) ) =  \frac{1}{\varepsilon^2} \frac{(cn^{\kappa})^{2-2\gamma} \, n(1-cn^{\kappa-1}) }{(1-cn^{\kappa -1})^{2\gamma} n^{3-2\gamma}} \Var ( h_2(\xi_1) ) \\
=&  \frac{1}{\varepsilon^2} \frac{(cn^{\kappa})^{2-2\gamma} }{n^{2-2\gamma}} (1-cn^{\kappa-1})^{1-2\gamma}  \Var ( h_2(\xi_1) ) = \frac{1}{\varepsilon^2} (cn^{\kappa-1})^{2-2\gamma} (1-cn^{\kappa-1})^{1-2\gamma}  \Var ( h_2(\xi_1) ).
\end{align*}
This converges to zero for $n\rightarrow \infty,$ as $\Var ( h_2(\xi_i))$ is constant and as $(\kappa-1)(2-2\gamma)=-1<0$. For the third term, we use analogous arguments as in the proof of Lemma \ref{lem:jn-conv}. Here we have $ (\frac{k}{n} (1-\frac{k}{n}))^\gamma\geq n^{-\gamma}(1-\frac{k_n^\ast}{n})^{\gamma} = (1-cn^{\kappa-1})^{\gamma} n^{-\gamma},$ 
and hence
\[
 \frac{1}{\big(\frac{k}{n}(1-\frac{k}{n} ) \big)^\gamma n^{3/2}} \leq \frac{n^{\gamma-3/2}}{(1-cn^{\kappa-1})^{\gamma}} .
\]
We apply Doob's maximal inequality and obtain
\begin{align*}
& P\Big(\max_{1\leq k\leq k_n^\ast} \frac{1}{\big(\frac{k}{n}(1-\frac{k}{n} ) \big)^\gamma n^{3/2}}  \big|\sum_{i=1}^k \sum_{j=k_n^\ast}^n \Psi(\xi_i,\xi_j)\big| \geq \varepsilon  \Big)   \\
\leq & P\Big(\max_{1\leq k\leq k_n^\ast}  \frac{n^{\gamma-3/2}}{(1-cn^{\kappa-1})^{\gamma}} \big| \sum_{i=1}^k \sum_{j=k_n^\ast}^n \Psi(\xi_i,\xi_j)\big| \geq \varepsilon  \Big) \\
 \leq & P\Big(\max_{1\leq k\leq k_n^\ast}  \big| \sum_{i=1}^k \sum_{j=k_n^\ast}^n \Psi(\xi_i,\xi_j)\big| \geq  \varepsilon \frac{(1-cn^{\kappa-1})^{\gamma}}{n^{\gamma-3/2}} \Big) 
 \leq  \frac{n^{2\gamma-3}}{\varepsilon^2 (1-cn^{\kappa-1})^{2\gamma}} \Var\Big(  \sum_{i=1}^{k_n^*} \sum_{j=k_n^\ast}^n \Psi(\xi_i,\xi_j)\Big)  \\
 = & \frac{n^{2\gamma-3}}{\varepsilon^2 (1-cn^{\kappa-1})^{2\gamma}} k_n^\ast (n-k_n^\ast) \Var(\Psi(\xi_1,\xi_2)) 
 =  \frac{n^{2\gamma-3} \, cn^{\kappa+1} (1-cn^{\kappa-1})}{\varepsilon^2 (1-cn^{\kappa-1})^{2\gamma}}  \Var(\Psi(\xi_1,\xi_2)) \\
= & \frac{c }{\varepsilon^2 }(1-cn^{\kappa-1})^{1-2\gamma} \, n^{2\gamma+\kappa-2} \Var(\Psi(\xi_1,\xi_2)) \longrightarrow 0, 
\end{align*}
for $n\rightarrow \infty$, as $2\gamma+\kappa-2<0$ and $\kappa-1<0$. Altogether we have shown that \[\max_{1\leq k\leq k_n^\ast} \frac{1}{(\frac{k}{n}(1-\frac{k}{n}))^\gamma} |J_n(k)-\E J_n(k)| \rightarrow 0.\] 

In an analogous way, we show that $\max_{ k_n^\ast \leq k< n} \frac{1}{(\frac{k}{n}(1-\frac{k}{n}))^\gamma} |J_n(k)-\E J_n(k)| \rightarrow 0$. For $k_n^* \leq k \leq n-1$, we have $J_n(k) =\frac{1}{n^{3/2}} \sum_{i=1}^{k_n^*} \sum_{j=k+1}^n h(\xi_i, \xi_j)$ and 
\begin{align*}
&\frac{1}{\big(\frac{k}{n}(1-\frac{k}{n}) \big)^{\gamma}} \big( J_n(k) - \E(J_n(k)) \big) \\
=&\frac{1}{\big(\frac{k}{n}(1-\frac{k}{n}) \big)^{\gamma}\, n^{3/2}}  \Big( (n-k) \sum_{i=1}^{k_n^*} h_1(\xi_i) +k_n^* \sum_{i=k+1}^n h_2(\xi_i) +  \sum_{i=1}^{k_n^*} \sum_{j=k+1}^n \Psi (\xi_i, \xi_j) \Big).
\end{align*}
Regarding the coefficient of the first term, we obtain for $k_n^* \leq k \leq n-1$
\begin{align*}
&\frac{1}{\big(\frac{k}{n}(1-\frac{k}{n}) \big)^{\gamma} \, n^{3/2}}   (n-k) 
= \frac{n^{\gamma}}{k^{\gamma}} \frac{n^{\gamma}}{(n-k)^{\gamma}} \frac{1}{n^{3/2}} (n-k) 
= \frac{n^{2\gamma-3/2}}{k^{\gamma}} (n-k)^{1-\gamma} \\
&\leq \frac{n^{2\gamma-3/2}}{(k_n^*)^{\gamma}} n^{1-\gamma} 
= \frac{n^{2\gamma-3/2}}{(cn^{\kappa})^{\gamma}} n^{1-\gamma} 
=\frac{1}{c^{\gamma}} n^{\gamma-\kappa \gamma-1/2}.
\end{align*}
Thus, together with Chebyshev's inequality we obtain
\begin{align*}
& \P \Big(\max_{k_n^* \leq k < n } \frac{1}{\big(\frac{k}{n}(1-\frac{k}{n}) \big)^{\gamma} \, n^{3/2}}   (n-k)\Big| \sum_{i=1}^{k_n^*} h_1(\xi_i)\Big|  \geq \varepsilon \Big)  
\leq  \P \Big(\frac{1}{c^{\gamma}} n^{\gamma-\kappa \gamma-1/2} \Big|\sum_{i=1}^{k_n^*} h_1(\xi_i) \Big|\geq \varepsilon \Big)   \\
&\leq \frac{1}{\varepsilon^2} \frac{1}{c^{2\gamma}} n^{2\gamma-2\kappa \gamma-1}  \Var \Big(\Big| \sum_{i=1}^{k_n^*} h_1(\xi_i)\Big| \Big)
 = \frac{1}{\varepsilon^2} \frac{1}{c^{2\gamma}} n^{2\gamma-2\kappa \gamma-1} k_n^*  \Var ( h_1(\xi_1) ) \\
& = \frac{1}{\varepsilon^2} \frac{1}{c^{2\gamma-1}} n^{2\gamma-2\kappa \gamma-1+\kappa }  \Var ( h_1(\xi_1) ).
\end{align*}
This converges to zero for $n \rightarrow \infty $, as $2\gamma-2\kappa \gamma-1+\kappa <0$ for $\gamma<1/2$ and $\kappa<1$. Regarding the second and third term, note that for $k_n^* \leq k \leq n-1$
\begin{align*}
\Big(\frac{k}{n}\big(1-\frac{k}{n}\big) \Big)^{\gamma} \geq \Big(\frac{k_n^*}{n} \cdot \frac{1}{n}\Big)^{\gamma} =c^{\gamma} n^{\kappa \gamma-2\gamma}
\end{align*}
and hence
\begin{align}\label{ineq:weight}
\frac{1}{\big(\frac{k}{n}(1-\frac{k}{n}) \big)^{\gamma}\, n^{3/2}} 
\leq \frac{1}{c^{\gamma} n^{\kappa \gamma-2\gamma} \, n^{3/2}} 
= \frac{1}{c^{\gamma}} n^{2\gamma-\kappa \gamma -3/2}.
\end{align}
Then we get with Kolmogorov's maximal inequality
\begin{align*}
& \P \Big( \max_{k_n^* \leq k < n } \frac{1}{\big(\frac{k}{n}(1-\frac{k}{n}) \big)^{\gamma} \, n^{3/2}} k_n^* \Big| \sum_{i=k+1}^n h_2(\xi_i) \Big|\geq \varepsilon  \Big) \\
&\leq \P \Big( \max_{k_n^* \leq k < n } \frac{1}{c^{\gamma}} n^{2\gamma-\kappa \gamma -3/2} k_n^* \Big| \sum_{i=k+1}^n h_2(\xi_i) \Big| \geq \varepsilon \Big) \\
&= \P \Big( \max_{k_n^* \leq k < n } c^{1-\gamma} n^{2\gamma-\kappa \gamma -3/2+\kappa}  \Big| \sum_{i=k+1}^n h_2(\xi_i) \Big| \geq \varepsilon  \Big) \\
&= \P \Big( \max_{1 \leq k \leq n-k_n^* } c^{1-\gamma} n^{2\gamma-\kappa \gamma -3/2+\kappa}  \Big| \sum_{i=1}^{k} h_2(\xi_{i}) \Big| \geq \varepsilon  \Big)\\ 
& \leq \frac{1}{\varepsilon^2} c^{2-2\gamma}\, n^{4\gamma-2\kappa \gamma -3+2\kappa} \Var \Big(\Big| \sum_{i=1}^{n-k_n^*} h_2(\xi_i)\Big| \Big) \\
&= \frac{1}{\varepsilon^2} c^{2-2\gamma} \,  n^{4\gamma-2\kappa \gamma -3+2\kappa} (n-k_n^*) \Var  (h_2(\xi_1)) \\
&= \frac{1}{\varepsilon^2} c^{2-2\gamma} (1-cn^{\kappa-1}) \, n^{4\gamma-2\kappa \gamma -2+2\kappa} \Var  (h_2(\xi_1)).
\end{align*}
This converges to zero for $n \rightarrow \infty $, as $4\gamma-2\kappa \gamma -2+2\kappa <0$ for $\gamma<1/2$. Regarding the last term, we use \eqref{ineq:weight} and Doob's maximal inequality to obtain
\begin{align*}
&\P\Big( \max_{k_n^* \leq k < n }  \frac{1}{\big(\frac{k}{n}(1-\frac{k}{n}) \big)^{\gamma}\, n^{3/2}} \Big| \sum_{i=1}^{k_n^*} \sum_{j=k+1}^n \Psi (\xi_i, \xi_j) \Big| \geq \varepsilon  \Big) \\
 \leq & \P\Big( \max_{k_n^* \leq k < n } \frac{1}{c^{\gamma}} n^{2\gamma-\kappa \gamma -3/2} \Big| \sum_{i=1}^{k_n^*} \sum_{j=k+1}^n \Psi (\xi_i, \xi_j) \Big| \geq \varepsilon  \Big) \\
 = & \P\Big( \max_{1 \leq k < n-k_n^* } \frac{1}{c^{\gamma}} n^{2\gamma-\kappa \gamma -3/2} \Big| \sum_{i=1}^{k_n^*} \sum_{j=1}^k \Psi (\xi_i, \xi_{k_n^*+j}) \Big| \geq \varepsilon  \Big) \\
 \leq & \frac{1}{\varepsilon^2c^{2 \gamma}} n^{4\gamma-2\kappa \gamma -3}  \Var \Big| \sum_{i=1}^{k_n^*} \sum_{j=1}^{n-k_n^*} \Psi (\xi_i, \xi_{k_n^*+j}) \Big|  \\
=& \frac{1}{\varepsilon^2c^{2 \gamma}} n^{4\gamma-2\kappa \gamma -3} k_n^*(n-k_n^*) \Var(\Psi (\xi_1, \xi_{k_n^*+1})) \\
=& \frac{1}{\varepsilon^2c^{2 \gamma}} n^{4\gamma-2\kappa \gamma -3} cn^{\kappa}(n-cn^{\kappa}) \Var(\Psi (\xi_1, \xi_{k_n^*+1})) \\
=& \frac{1}{\varepsilon^2c^{2 \gamma-1}} (1-cn^{\kappa-1}) n^{4\gamma-2\kappa \gamma -2+\kappa} \Var(\Psi (\xi_1, \xi_{k_n^*+1})).
\end{align*}
This goes to zero for $n\rightarrow \infty$, as $4\gamma-2\kappa \gamma -2+\kappa<0$ for $\gamma<1/2$.
\end{proof}

\section{Auxiliary results from the literature}
\begin{lemma}[\cite{HR:1955}] \label{H-R-inequality}
Let $\{ \varepsilon_i,~ 1\leq i \leq n  \}$ be independent random variables with finite second moments and let $\{ b_i, ~ i\geq 1\}$ be a positive, decreasing real sequence. Then for any $\alpha>0$ 
\begin{align*}
P\Big( \max_{1\leq k \leq n } b_k \Big| \sum_{i=1}^k (\varepsilon_i-E(\varepsilon_i))\Big| \geq \alpha \Big) \leq \frac{1}{\alpha^2} \sum_{i=1}^n \Var(\varepsilon_i)b_i^2 
\end{align*}
\end{lemma}

\bibliographystyle{abbrvnat.bst}
\bibliography{literature.bib}

\begin{thebibliography}{17}
\providecommand{\natexlab}[1]{#1}
\providecommand{\url}[1]{\texttt{#1}}
\expandafter\ifx\csname urlstyle\endcsname\relax
  \providecommand{\doi}[1]{doi: #1}\else
  \providecommand{\doi}{doi: \begingroup \urlstyle{rm}\Url}\fi

\bibitem[Berkes et~al.(2009)Berkes, Gombay, and Horv{\'a}th]{BGH:2009}
I.~Berkes, E.~Gombay, and L.~Horv{\'a}th.
\newblock Testing for changes in the covariance structure of linear processes.
\newblock \emph{Statist. Plann. Inference}, 139\penalty0 (6):\penalty0
  2044--2063, 2009.

\bibitem[Billingsley(1999)]{B:1999}
P.~Billingsley.
\newblock \emph{{Convergence of Probability Measures}}.
\newblock Wiley Series in Probability and Statistics: Probability and
  Statistics. John Wiley \& Sons, Inc., New York, second edition, 1999.
\newblock ISBN 0-471-19745-9.
\newblock \doi{10.1002/9780470316962}.
\newblock URL \url{https://mathscinet.ams.org/mathscinet-getitem?mr=1700749}.
\newblock A Wiley-Interscience Publication.

\bibitem[Cs\"{o}rg\H{o} and Horv\'{a}th(1997)]{CsH:1997}
M.~Cs\"{o}rg\H{o} and L.~Horv\'{a}th.
\newblock \emph{{Limit Theorems in Change-Point Analysis}}.
\newblock Wiley Series in Probability and Statistics. John Wiley \& Sons, Ltd.,
  Chichester, 1997.
\newblock ISBN 0-471-95522-1.
\newblock URL \url{https://mathscinet.ams.org/mathscinet-getitem?mr=2743035}.

\bibitem[Cs\"{o}rg\H{o} et~al.(2008)Cs\"{o}rg\H{o}, Szyszkowicz, and
  Wang]{CsSW:2008}
M.~Cs\"{o}rg\H{o}, B.~Szyszkowicz, and Q.~Wang.
\newblock {Asymptotics of Studentized U-type processes for changepoint
  problems}.
\newblock \emph{Acta Math. Hungar.}, 121\penalty0 (4):\penalty0 333--357, 2008.

\bibitem[Dehling et~al.(2013)Dehling, Rooch, and Taqqu]{DRT:2013}
H.~Dehling, A.~Rooch, and M.~Taqqu.
\newblock Non-parametric change-point tests for long-range dependent data.
\newblock \emph{Scandinavian Journal of Statistics}, 40\penalty0 (1):\penalty0
  153--173, 2013.

\bibitem[Dehling et~al.(2015)Dehling, Fried, Garcia, and Wendler]{DFGW:2015}
H.~Dehling, R.~Fried, I.~Garcia, and M.~Wendler.
\newblock {Change-Point Detection Under Dependence Based on Two-Sample
  U-Statistics}.
\newblock In D.~Dawson, R.~Kulik, M.~Ould~Haye, B.~Szyszkowicz, and Y.~Zhao,
  editors, \emph{Asymptotic Laws and Methods in Stochastics}, volume~76, pages
  195--220. Springer, New York, NY, 2015.

\bibitem[Dehling et~al.(2017)Dehling, Rooch, and Taqqu]{DRT:2017}
H.~Dehling, A.~Rooch, and M.~Taqqu.
\newblock Power of change-point tests for long-range dependent data.
\newblock \emph{Electronic Journal of Statistics}, 11\penalty0 (1):\penalty0
  2168--2198, 2017.

\bibitem[Dehling et~al.(2022)Dehling, Vuk, and Wendler]{DVW:2022}
H.~Dehling, K.~Vuk, and M.~Wendler.
\newblock Change-point detection based on weighted two-sample u-statistics.
\newblock \emph{Electronic Journal of Statistics}, 16\penalty0 (1):\penalty0
  862--891, 2022.

\bibitem[Ferger(1994)]{F:1994}
D.~Ferger.
\newblock {On the power of nonparametric changepoint-tests}.
\newblock \emph{Metrika}, 41\penalty0 (1):\penalty0 277--292, 1994.

\bibitem[Gombay(2000)]{G:2000b}
E.~Gombay.
\newblock Comparison of u-statistics in the change-point problem and in
  sequential change detection.
\newblock \emph{Period. Math. Hungar.}, 41:\penalty0 157--166, 2000.

\bibitem[H{\'a}jek and R{\'e}nyi(1955)]{HR:1955}
J.~H{\'a}jek and A.~R{\'e}nyi.
\newblock Generalization of an inequality of kolmogorov.
\newblock \emph{Acta Mathematica Hungarica}, 6\penalty0 (3-4):\penalty0
  281--283, 1955.

\bibitem[Horv{\'a}th et~al.(2020)Horv{\'a}th, Miller, and Rice]{HMR:2020}
L.~Horv{\'a}th, C.~Miller, and G.~Rice.
\newblock A new class of change point test statistics of r{\'e}nyi type.
\newblock \emph{Journal of Business \& Economic Statistics}, 38\penalty0
  (3):\penalty0 570--579, 2020.

\bibitem[Horv{\'a}th et~al.(2021)Horv{\'a}th, Rice, and Zhao]{HRZ:2021}
L.~Horv{\'a}th, G.~Rice, and Y.~Zhao.
\newblock Change point analysis of covariance functions: A weighted cumulative
  sum approach.
\newblock \emph{J. Multivariate Anal.}, page 104877, 2021.

\bibitem[Ra{\v{c}}kauskas and Wendler(2020)]{RW:2020}
A.~Ra{\v{c}}kauskas and M.~Wendler.
\newblock Convergence of u-processes in h{\"o}lder spaces with application to
  robust detection of a changed segment.
\newblock \emph{Statist. Papers}, 61\penalty0 (4):\penalty0 1409--1435, 2020.

\bibitem[Robbins et~al.(2011)Robbins, Gallagher, Lund, and Aue]{RGLA:2011}
M.~Robbins, C.~Gallagher, R.~Lund, and A.~Aue.
\newblock Mean shift testing in correlated data.
\newblock \emph{J. Time Series Anal.}, 32\penalty0 (5):\penalty0 498--511,
  2011.

\bibitem[Szyszkowicz(1991)]{S:1991}
B.~Szyszkowicz.
\newblock Changepoint problems and contiguous alternatives.
\newblock \emph{Statist. Probab. Lett.}, 11\penalty0 (4):\penalty0 299--308,
  1991.

\bibitem[Xie et~al.(2014)Xie, Li, and Xiong]{XLX:2014}
H.~Xie, D.~Li, and L.~Xiong.
\newblock Exploring the ability of the pettitt method for detecting change
  point by monte carlo simulation.
\newblock \emph{Stoch. Environ. Res. Risk Assess.}, 28:\penalty0 1643--1655,
  2014.

\end{thebibliography}

\end{document}